\documentclass[12pt]{amsart}

\usepackage{tikz-cd} 
\usepackage[margin=1in]{geometry}
\usepackage[english]{babel}
\usepackage[utf8x]{inputenc}
\usepackage{amsmath}
\usepackage{enumitem}
\usepackage{amstext}
\usepackage{amsfonts}
\usepackage{amssymb}
\usepackage{amsthm}
\usepackage{amsrefs}
\usepackage{xcolor}
\definecolor{revisionpurple}{RGB}{128,0,128}
\newcommand{\revision}[1]{{\color{revisionpurple}#1}}
\makeatletter
\let\revision@bibstart\bib@start
\def\bib@start#1{%
  \revision@bibstart{#1}%
  \ifcsname revision@bib@#1\endcsname
    \csname revision@bib@#1\endcsname
  \fi
}
\newcommand{\revision@date}{%
  \let\revision@printdate\print@date
  \def\print@date{\revision{\revision@printdate}}%
}
\expandafter\let\csname revision@bib@A19\endcsname\revision@date
\expandafter\let\csname revision@bib@ocneanu2006actions\endcsname\revision@date
\@namedef{revision@bib@ryo2021remark}{%
  \let\revision@printauthors\PrintAuthors
  \def\PrintAuthors##1{\revision{\revision@printauthors{##1}}}%
}
\makeatother
\usepackage{mathtools}
\usepackage{dsfont}
\usepackage{graphicx}
\usepackage[colorinlistoftodos]{todonotes}
\usepackage[hidelinks]{hyperref}
\usepackage{float}
\usepackage[colorinlistoftodos]{todonotes}
\usepackage[nameinlink,capitalise]{cleveref}

\usepackage{comment}
\allowdisplaybreaks
\newcommand\numberthis{\addtocounter{equation}{1}\tag{\theequation}}
\newtheorem{theorem}{Theorem}[section]
\newtheorem{lemma}[theorem]{Lemma}
\newtheorem{prop}[theorem]{Proposition}
\newtheorem{cor}[theorem]{Corollary}
\newtheorem{thm}[theorem]{Theorem}
\newtheorem{lem}[theorem]{Lemma}

\newtheorem*{cor*}{Corollary}
\newtheorem*{thm*}{Theorem}
\newtheorem*{lem*}{Lemma}

\newtheorem*{prop*}{Proposition}
\theoremstyle{definition}
\newtheorem{definition}[theorem]{Definition}
\newtheorem{defn}[theorem]{Definition}
\newtheorem{example}[theorem]{Example}
\newtheorem*{defn*}{Definition}

\theoremstyle{remark}
\newtheorem{remark}[theorem]{Remark}

\newcommand{\Ad}{\operatorname{Ad}}
\newcommand{\cA}{ A }
\newcommand{\cB}{ B }

\newcommand{\cM}{M}
\newcommand{\cN}{N}

\newcommand{\cZ}{\mathcal{Z}}

\newcommand{\bC}{{\mathbb{C}}}

\newcommand{\bN}{{\mathbb{N}}}
\newcommand{\bZ}{{\mathbb{Z}}}

\newcommand{\C}{\mathbb C}
\newcommand{\G}{\Gamma}

\newcommand{\A}{A}
\newcommand{\B}{B}

\newcommand{\aut}{\mathrm{Aut}}

\title{Freeness conditions for intermediate $C^*$-algebras of crossed 
products}
\author{Tattwamasi Amrutam}
\address{Ben Gurion University of the Negev.
	Department of Mathematics.
	Be'er Sheva, 8410501, Israel and Institute of Mathematics of the Polish Academy of Sciences, ul. Sniadeckich 8, 00-656, Warszawa, Poland.
}
\email{tattwamasiamrutam@gmail.com, tattwama@post.bgu.ac.il}
\author{Ilan Hirshberg}
\address{Ben Gurion University of the Negev.
	Department of Mathematics.
	Be'er Sheva, 8410501, Israel.
}

\email{ilan@bgu.ac.il}
\author{Apurva Seth}
\address{Ben Gurion University of the Negev.
	Department of Mathematics.
	Be'er Sheva, 8410501, Israel and
	Mathematical Institute,
	University of Oxford,
	Andrew Wiles Building,
	Radcliffe Observatory Quarter,
	Woodstock Road,
	Oxford,
	OX2 6GG, UK.
}
\email{apurvaseth14@gmail.com}
\date{\today}

\thanks{This research was supported by the European Research Council (ERC) 
under the
	European Union’s Seventh Framework Programme (FP7-2007-2013) (Grant 
	agreement No. 101078193), 
	by the Israel Science Foundation (Grant No. 1175/18), by the US-Israel
	Binational Science Foundation, by the Engineering and Physical Sciences Research Council (EP/X026647/1), 
	 and by the Fields Institute. The first author was also supported by National Science Centre, Poland, Sonata grant number 2025/59/D/ST1/03117.}

\begin{document}
\begin{abstract}
Let $B$ be a separable $C^*$-algebra, let $\Gamma$ be a discrete 
countable group, let $\alpha \colon \Gamma \to \aut(B)$ be an action, and let 
$A$ be an invariant subalgebra. We find freeness conditions that 
guarantee that any 
intermediate $C^*$-algebra $A \rtimes_{\alpha,r} \Gamma \subseteq C \subseteq B 
\rtimes_{\alpha,r} \Gamma$ is a crossed product of an intermediate invariant 
subalgebra $A \subseteq B_1 \subseteq B$ by $\Gamma$. These are used to 
generalize related results by Suzuki.
\end{abstract}
\maketitle
\section{Introduction}
Let $\Gamma$ be a discrete group. Let $\B$ be a $\Gamma$-$C^*$-algebra, that 
is, a $C^*$-algebra along with an action $\alpha$ of $\Gamma$ on it by 
automorphisms. If 
$\A \subseteq \B$ is a $\Gamma$-invariant subalgebra, then the restriction of 
$\alpha$ induces an inclusion  $A \rtimes_{\alpha,r} \Gamma \subseteq \B \rtimes_{\alpha,r}  \Gamma$.
We continue to write $\alpha$ for the restriction to $A$. Suppose we have an intermediate 
$C^*$-algebra 
$\A \rtimes_{\alpha,r}  
\Gamma \subseteq C \subseteq B \rtimes_{\alpha,r}  \Gamma$. The goal of this paper is to find conditions ensuring
$C = B_1 \rtimes_{\alpha,r}  \Gamma$
for some
$\Gamma$-invariant subalgebra $A \subseteq B_1 \subseteq B$.

Related questions were studied under various other hypotheses in 
\cite{Suz, AK,A19, ryo2021remark,rordam2021irreducible}. In 
\cite{A19}, the first named author considers the case in which 
$A = \C$, the group $\Gamma$ is $C^*$-simple, and the kernel of the action of 
$\Gamma$ on $A$ is sufficiently large. In \cite{AK}, the first named author and 
Kalantar studied the case in which $A = \C$, $B = C(X)$, the group $\Gamma$ 
is $C^*$-simple, and the action of $\Gamma$ on $X$ is minimal, and the result is 
that intermediate subalgebras are simple. In \cite{rordam2021irreducible}, one  
considers, among other things, the case in which $A = \C$, the group $\Gamma$ is 
$C^*$-simple, and finds conditions for the inclusion to be irreducible. In 
\cite{Suz}, Suzuki addresses the case in which $A = C(X)$ and $\Gamma$ acts 
freely on it, as well as the von Neumann algebra analog. He also considers 
the case in which $A$ and $B$ are noncommutative von Neumann algebras, as well 
as certain actions on $C^*$-algebras, including some noncommutative Bernoulli 
shifts; see also \cite{ryo2021remark} for more on Suzuki's 
conditions. For somewhat related work concerning the structure of intermediate 
$C^*$-algebras and von Neumann algebras of the form $A \subseteq C \subseteq A 
\rtimes_{\alpha,r}  \Gamma$, we refer the reader to 
\cite{choda1978galois,izumi1998galois,Cameron2015BimodulesIC,cameron2016intermediate}.

 Our main theorem in this paper is the following. The definitions of 
 the 
 conditions in the theorem are provided in Section \ref{Sec: Prelim}. 
\begin{theorem}\label{mainintmresult}
	Let \(\Gamma\) be a countable discrete group with the approximation 
	property (AP).  
	Let \(B\) be a separable \(C^*\)-algebra and let
	\(\alpha\colon\Gamma\to\aut(B)\) be an action.  
	Let \(A\subseteq B\) be a \(\Gamma\)-invariant subalgebra containing an 
	approximate identity for \(B\).  
	Fix a free ultrafilter \(\omega\) on \(\mathbb N\).  Assume that for every 
	central $\alpha^{**}$-invariant projection $P\in B^{**}$ such that 
	{$PB^{**}P$ has separable predual}, we have 
	\[
	(PAP)' \cap (PA^{**}P)^\omega
	\;=\;
	(PBP)' \cap (PA^{**}P)^\omega
	,
	\]
	and that \(\alpha^{**}\colon\Gamma\to\aut(PA^{**}P)\) is centrally free.  
	Then every intermediate \(C^*\)-algebra
	\(\,A\rtimes_{\alpha,r}\Gamma\subseteq C\subseteq 
	B\rtimes_{\alpha,r}\Gamma\)
	is of the form \(B_1\rtimes_{\alpha,r}\Gamma\), where
	$A\subseteq B_1\subseteq B$ is a $\Gamma$-invariant $C^*$-subalgebra.
\end{theorem}

While the statement of \cref{mainintmresult} may seem a bit technical, the  
significance of the paper lies in its innovative approach. The method employed 
involves the interplay between von Neumann algebras and $C^*$-algebras, 
specifically the application of  Rokhlin-type theorems for the induced action 
on the double dual. This approach provides an alternative proof for Suzuki's 
results when $A$ and $B$ are abelian, and allows us to push it further, for 
instance to cover cases in which $A$ is type I and where the techniques from 
\cite{Suz} do not apply; those applications are
discussed in Section~\ref{Examples}.

{Let us describe how our approach compares to that of
\cite{Suz}. The main theorem of \cite{Suz} in the $C^*$-algebra setting
requires the inclusion $A\subseteq B$ to be centrally $\Gamma$-free
(\cite[Definition~4.1]{Suz}; the definition is recalled in \cref{def_centfree}
below): for any $b,c\in B$, any $\varepsilon>0$ and any
$s\in\Gamma\setminus\{e\}$, one needs to exhibit an element $\xi$ in the unit
ball of $\ell^{2}(A)$ such that $\lVert\xi^*b\xi-b\rVert<\varepsilon$ and
$\lVert\xi^*c\,s(\xi)\rVert<\varepsilon$, where $s(\xi)$ denotes the entrywise
action of $s$ on $\xi$. In other words, the towers witnessing freeness have to
be constructed inside $A$ with respect to the operator norm. The hypotheses of
\cref{mainintmresult}, in contrast, are placed on the induced action on the
double dual $A^{**}$, and only need to be verified one normal state at a time;
this gives us access to von Neumann algebraic techniques, and makes the
hypotheses insensitive to the topology of $\hat{A}$. As a concrete example, if
$\Gamma$ is a finite group, $D$ is a separable unital $C^*$-algebra, $A$ is a
separable type I $C^*$-algebra on whose spectrum $\hat{A}$ the induced action
is free, and $B=A\otimes D$, then every intermediate subalgebra $A
\rtimes_{\alpha,r} \Gamma \subseteq C \subseteq B \rtimes_{\alpha,r} \Gamma$ is
a crossed product (see \cref{cor_examples}); since $\hat{A}$ may be very far
from being Hausdorff in this generality, we do not know how to verify the
central $\Gamma$-freeness of such inclusions, and the techniques of \cite{Suz}
do not seem to apply.}

The assumption of property AP on the group $\Gamma$ cannot generally be relaxed, as shown in \cite{Suz}. Property AP (the \emph{approximation 
property}) was introduced by Haagerup-Kraus in \cite{haagerup1994approximation} 
in which it was shown that every weakly amenable group has property AP (see 
\cite[Theorem~1.12]{haagerup1994approximation}). In particular, every amenable 
group, being weakly amenable, has the approximation property.  
Hyperbolic groups also have the approximation property~\cite{ozawa2008weak}. 
Roughly speaking, property AP allows us to recover enough information on 
crossed-product elements by considering Fourier coefficients. We refer the 
reader to \cite[Chapter~12]{BroOza08} for further details on the approximation 
property.

{A key technical ingredient of this paper, which may be of
independent interest, is a collection of strengthened versions of the
noncommutative Rokhlin theorems of Connes~\cite{connes1975outer} and of
Ocneanu~\cite{ocneanu2006actions}, proved in
Section~\ref{relativizedrokhlinconnes}. Connes' Rokhlin theorem
\cite[Theorem~1.2.5]{connes1975outer} concerns single automorphisms of finite
von Neumann algebras admitting an invariant trace; we prove a version for free
actions of $\bZ$ which does not assume that the trace is invariant
(\cref{prop:connes}). Similarly, we strengthen the conclusion of Ocneanu's
Rokhlin theorem \cite[Theorem~6.1]{ocneanu2006actions} for centrally free
actions of amenable groups, by showing that if the group admits large
monotiles, then the towers can be indexed by a single prescribed monotile
(\cref{prtunityinvna}). Building on these, we prove relative versions adapted
to an inclusion $A\subseteq B$ of separable $C^*$-algebras
(\cref{thm:relativizedconnes} and \cref{thm:relativizedocneanu}), in which the
Rokhlin towers consist of mutually orthogonal positive contractions in $A$
which are approximately equivariant, approximately projections and
approximately central with respect to $B$, in the $*$-strong topology coming
from the double dual $B^{**}$. These relative versions are the ones used in
the proof of \cref{mainintmresult}.}

We now present a brief outline of the proof. Let
$D:=B\rtimes_{\alpha,r}\Gamma$, and let $\mathbb{E}\colon D\to B$ be the
canonical conditional expectation. Property AP reduces the proof to showing
that $\mathbb{E}(C)\subseteq C$. Fix $x\in C$. The Rokhlin towers are first
constructed in $A^{**}\subseteq B^{**}$ and then approximated by mutually
orthogonal positive contractions in $A$. The resulting completely positive
contractions preserve $C$.

The paper is structured as follows: we start by reviewing some preliminary 
notions. Next, in Section~\ref{relativizedrokhlinconnes}, we prove the required 
versions of the noncommutative Rokhlin Theorems of Connes 
(\cref{thm:relativizedconnes}) and of Ocneanu 
(\cref{thm:relativizedocneanu}). In Section~\ref{maintheoremproof}, we present 
the proof of our main result, \cref{mainintmresult}. In 
Section~\ref{Examples}, we provide some natural examples of inclusions and 
actions that satisfy the hypotheses of \cref{mainintmresult}. Finally, we 
introduce the concept of relative Rokhlin dimension for the inclusion 
$A\subseteq B$ of $\Gamma$-$C^*$-algebras in Section~\ref{sec_rokhlin} and show 
that the assumption of pointwise finite relative Rokhlin dimension implies that 
the inclusion is centrally $\Gamma$-free according to \cite{Suz}.

\subsection*{Acknowledgements}
We thank N. Christopher Phillips for his comments on an earlier draft of this 
paper. We also thank the referees for their useful comments and suggestions. 

\section{Preliminaries}
\label{Sec: Prelim}
Throughout this paper, $\Gamma$ will denote a discrete countable group. We 
denote the cardinality of a set $S$ by $|S|$. If $S\subset\Gamma$ is finite, we 
sometimes write $S\subset\subset \Gamma$. 

\subsection{Monotileable amenable groups and paving systems}
We begin by reviewing the definition of monotileable groups from \cite{weiss2001monotileable}.

\begin{defn}\label[definition]{def:monotileable}
A \emph{monotile} $T$ in a discrete group $\G$ is a finite set such that 
\begin{enumerate}
\item $\exists$ $C\subset\Gamma$ such that $\{cT: c\in C\}$ is a covering of $\G$.
\item $cT\cap c'T=\emptyset$, for any  $c,c'\in C$ such that $c\neq c'$.
\end{enumerate}
\end{defn}
A discrete group $\G$ is said to be \emph{monotileable} if it admits a 
F{\o}lner sequence consisting of monotiles. To our knowledge, it is still not 
known whether all amenable groups are monotileable. In this paper, we need a 
much weaker property. We say that $\Gamma$ \emph{admits large monotiles} if for 
any finite subset $F \subset \Gamma$ there exists a monotile $T \subset \Gamma$ 
such that $F \subset T$. We do not know whether there exist groups that do not 
satisfy this property. For example, it is easy to see that any residually 
finite group admits large monotiles.

We briefly recall some definitions concerning paving systems; see 
\cite[Chapter~3]{ocneanu2006actions} for more details.
\begin{defn}
For $K\subset\subset \G$ and $\delta>0$, we say that a non-empty set $S$ of 
$\G$ is $(\delta,K)$-\emph{(right) invariant} if it is finite and $\lvert 
S\cap\bigcap_{g\in K}Sg\rvert>(1-\delta)\lvert S \rvert$.
\end{defn}
\begin{defn}
A system $\{K_i\}_{i\in I}$ of finite sets in $\G$ is said to be 
$\varepsilon$-disjoint for $\varepsilon>0$, if there exist subsets 
$K_i'\subseteq 
K_i$ for all 
$i\in I$, such that $\lvert K_i'\rvert\geq (1-\varepsilon)\lvert K_i\rvert$ and 
$(K_i')_{i \in I}$ are disjoint.
\end{defn}
\begin{defn}
A system $K_1,K_2,\ldots,K_N$ of finite subsets of $\G$ 
\emph{$\varepsilon$-paves} 
the finite subset $S$ of $\G$ if there are subsets $B_1,B_2,\ldots,B_N$ of 
$\G$, called the \emph{paving centers}, such that 
\begin{enumerate}
\item {$\displaystyle\bigcup_{i=1}^N B_iK_i\subset S$.}
\item {$\displaystyle\left\lvert S\setminus\bigcup_{i=1}^N B_iK_i
\right\rvert<\varepsilon\lvert S\rvert$.}
\item {$(bK_i)_{b\in B_i,\,1\leq i\leq N}$ are
$\varepsilon$-disjoint.}
\end{enumerate}
\end{defn}
\begin{defn}
We call $\{K_1,K_2,\ldots,K_N\}$ an \emph{$\varepsilon$-paving system of sets}, 
if 
there exist $\delta > 0$ and $F\subset\subset \G$  such that 
$\{K_1,K_2,\ldots,K_N\}$ 
$\varepsilon$-pave any $(\delta,F)$-invariant subset $S$ of $\G$.
\end{defn}
 We conclude this discussion with the following  observation concerning 
 paving systems in amenable groups that admit large monotiles.

\begin{lem}\label[lemma]{prop:monotileableepsilon}
	Let $\G$ be a group, and let $T$ be a monotile in $\G$. Let $\varepsilon>0$. 
	Then $T$ $\varepsilon$-paves any finite $(\varepsilon,T^{-1}T)$-invariant 
	subset $S$ of $\G$. 
\begin{proof}
Since $T$ is a monotile, by \cref{def:monotileable}, there exists $C\subseteq 
\G$ such that $\{cT:c\in C\}$ forms a disjoint covering of $\G$.
Let $S$ be an {$(\varepsilon,T^{-1}T)$}-invariant subset of $\G$. Define 
$C_0:=\{c\in C : cT\subseteq S\}$. Then $C_0T\subseteq S$. Define 
$\gamma:\G\to C$ by requiring that $\gamma(g)$ is the unique element in $C$ 
such that 
$g\in \gamma(g)T$. Then, for $g\in S\setminus C_0T$, we have $g\in S$ and 
$g\notin C_0T$, hence $\gamma(g)T\nsubseteq S$, because if $\gamma(g)T\subseteq 
S$, 
then $\gamma(g)\in C_0$ and so $g=\gamma(g)t\in\gamma(g)T\in C_0T$, which is a 
contradiction. Hence, there exists some $t'\in T$ such that $\gamma(g)t'\notin 
S$. Setting $t = \gamma(g)^{-1}g \in T$, this implies 
$gt^{-1}t'=\gamma(g)t'\notin S$, so that $g\notin 
\bigcap_{h\in T^{-1}T}Sh$. Thus, $S\setminus C_0T\subseteq 
S\setminus\bigcap_{h\in T^{-1}T}Sh $. Now, since $S$ is taken to be 
$(\varepsilon, 
T^{-1}T)$-invariant, we have {$\lvert\bigcap_{h\in T^{-1}T}Sh\rvert>(1-\varepsilon)\lvert S\rvert$}, so that $\lvert S\setminus 
C_0T\rvert\leq\lvert S\setminus \bigcap_{h\in T^{-1}T} Sh 
\rvert<\varepsilon\lvert 
S\rvert$. Thus, $T$ $\varepsilon$-paves any $(\varepsilon, T^{-1}T)$-invariant 
subset 
$S$ of $\G$.
\end{proof}
\end{lem}
\subsection{Ultraproduct von Neumann Algebras and $C^*$-to-$W^*$-Ultraproducts}
We briefly recall the definitions for ultraproduct von Neumann algebras; see 
\cite[Chapter~5]{ocneanu2006actions} for more details. Fix a free 
ultrafilter $\omega$ on $\bN$. Let $\cM$ be a von 
Neumann algebra with separable predual and let $\cM_*$ be its predual. For 
$a\in M$ and $\varphi\in 
\cM_*$, denote by $a\varphi,\, \varphi a,\, [a,\varphi]\in \cM_*$ the linear functionals
\[
a\varphi(x):=\varphi(ax),\,\, \varphi a(x):=\varphi(xa)\text{ and}\,\, 
[a,\varphi]=\varphi a-a\varphi.
\]
For $a\in \cM$ and $\varphi\in \cM_*^+$, we define the following seminorms
\[
\|a\|_\varphi:=\varphi(a^*a)^{1/2},\,\, \|a\|_\varphi^\#:=\varphi(a^*a+aa^*)^{1/2}.
\]
Consider $\ell^\infty(\bN,\cM)$, the $C^*$-algebra of norm bounded sequences in 
$\cM$. An element $(x^{ \nu } )_{ \nu } $ of $\ell^\infty (\bN,\cM)$ is said to 
be
\begin{enumerate}
\item $\omega$-trivial if $x^{ \nu }  \underset{{ \nu } \to 
\omega}{\longrightarrow} 0$  
*-strongly.
\item $\omega$-central if for any $\varphi\in \cM_*$, we have $\|[x^{ \nu } 
,\varphi]\| 
\underset{{ \nu } \to \omega}{\longrightarrow} 0$.
\end{enumerate}
Let $\mathcal{I}_\omega(M)$ and $\mathcal{C}_\omega(M)$ be the sets of 
$\omega$-trivial 
and $\omega$-central sequences in $\cM$, respectively. Those are $C^*$-subalgebras of 
$\ell^\infty(\bN,\cM)$. Let $ N_\omega(M)$ be the normalizer of 
$\mathcal{I}_\omega(M)$ in $\ell^\infty(\bN,\cM)$. Then the quotient 
$C^*$-algebras 
defined as $ M^\omega:= N_\omega(M)/\mathcal{I}_\omega(M)$ and 
$ M_\omega:= C _\omega(M)/\mathcal{I}_\omega(M)$ are von Neumann algebras. 
Notice that $\cM$ and 
$\cM_\omega$ are von Neumann subalgebras of  $\cM^\omega$ (where $\cM$ is 
identified with the image of the constant sequences in   
$\ell^\infty(\bN,\cM)$). For any state
$\varphi\in \cM_*$, we define a state  $\varphi^\omega$ on  
$ M^\omega$ by 
$\varphi^\omega((x^{ \nu } )_{ \nu } )=\lim_{{ \nu } \to 
\omega}\varphi(x^\nu)$. Its restriction 
to $ M_\omega$ will be denoted 
by $\varphi_\omega$. Note that $\varphi_{\omega}$ is a normal trace on 
$ M_{\omega}$, regardless of whether $\varphi$ is a trace or not. 
Any automorphism $\alpha\in\aut(\cM)$ induces automorphisms  
$\alpha^\omega$ and $\alpha_\omega$ of $\cM^\omega$ and $\cM_\omega$ 
respectively.  This follows from the fact that for any $\varphi\in\cM_*$, and 
for any $(x^{ \nu } )_{ \nu } \in\ell^\infty({\bN},\cM)$, we have
\[
\left\|\alpha(x^{ \nu } )\right\|_\varphi^2=\varphi\circ\alpha({x^{ \nu } 
}^*x^{ \nu } )=\left\|x^{ \nu } \right\|_{\varphi\circ\alpha}^2,\,\text{and}
\]
\[
{\left\|[\alpha(x^{\nu}),\varphi]\right\|
=\left\|[x^{\nu},\varphi\circ\alpha]\right\|}.
\]

The ultraproduct notation for $C^*$-algebras is different. Let $A$ be a 
$C^*$-algebra. We write $A_{\infty} = \ell^{\infty}(\bN,A)/C_0(\bN,A)$. One 
could replace $\infty$ with a free ultrafilter, but that will not be needed 
here. Denote the quotient map by $\pi_\infty:\ell^{\infty}(\bN,A)\to A_\infty$. 
Now,
 $A$ embeds in $A_{\infty}$ as the $\pi_\infty$-image of the constant sequences in 
 $\ell^{\infty}(\bN,A)$, and the central sequence algebra, denoted by $A_{\infty} \cap 
 A'$, is the $\pi_\infty$-image of central sequences in $\ell^{\infty}(\bN,A)$. 
 The two-sided annihilator $\text{Ann}(A,A_\infty)$ is the closed ideal of 
 $A_\infty\cap A'$ defined as the $\pi_\infty$-image of sequences $(a_n)_n\in 
 \ell^{\infty}(\bN,A)$ such that $\|a_n a\|+\|a 
 a_n\|\underset{{n\to\infty}}\longrightarrow 0 $, for all $a\in A$. Denote 
 $F(A):=(A_\infty\cap A' )/ \text{Ann}(A,A_\infty)$. Then any automorphism 
 $\alpha$ of $A$ induces an automorphism of $F(A)$, which we also denote by 
 slight abuse of notation by $\alpha$.
 
We now recall the notion of $C^*$-to-$W^*$-ultraproduct $(A,\rho)_\omega$, for a $C^*$-algebra $A$ and state $\rho$; see \cite{ando2016non} for more details. Let $L_\rho(A)$ be the closed left ideal of $\ell^{\infty}(\bN,A)$ consisting of those $(a_v)\in\ell^{\infty}(\bN,A)$ such that
 \[
 \lim_{v\to\omega}\rho(a_v^*a_v) = 0 \, .
 \]
Define $\mathcal{I}_\rho(A):=L_\rho(A)\cap L_\rho(A)^*$. Denote by 
$N(\mathcal{I}_\rho(A))$  the normalizer of $\mathcal{I}_\rho(A)$ in 
$\ell^{\infty}(\bN,A)$. The $C^*$-to-$W^*$-ultraproduct $(A,\rho)_\omega$ is 
defined to be the quotient $C^*$-algebra 
$N(\mathcal{I}_\rho(A))/\mathcal{I}_\rho(A)$.

Next, we consider inclusions $N\subset M$ of unital von Neumann algebras with
separable preduals, and let $B\subseteq M$ be a weakly dense
$C^*$-subalgebra. Let $\varphi$ be a faithful normal state on $M$. Note that we 
do not necessarily have $N^\omega\subset M^\omega$. We may, nonetheless, define 
\[
B'\cap N^\omega:=\{(x^v)\in N_\omega(N):\,\, \|x^vb-bx^v\|_\varphi^\#\underset{v\to\omega}\longrightarrow 0,\,\,\forall b\in B\}/\mathcal{I}_\omega
(N)
\, .
\]

We recall the following.
\begin{prop}\cite[Proposition 3.4, Proposition 4.15]{ando2016non}\label[proposition]{prop_ando_kirchberg}
Let $\varphi$ be a faithful normal state on a $W^*$-algebra ${N}$, and let $A\subset N$ be a weakly dense $C^*$-subalgebra. Set $\rho:=\varphi|_A$. Then, there exists a $*$-isomorphism from the $C^*$-to-$W^*$-ultraproduct $(A,\rho)_\omega$ onto $N^\omega$ which maps $A'\cap (A,\rho)_\omega$ onto $N'\cap N^\omega$.
\end{prop}

Throughout, a partition of unity in a von Neumann algebra means a finite
family of mutually orthogonal projections whose sum is the identity.

The following simple lemma lets us approximate projections in 
$ M_{\omega}$ with ones from $ M$, as is done in 
\cite[Lemma~1.1.5]{connes1975outer}. 
\begin{lem}\label[lemma]{lem:connes}
Let $\varphi$ be a faithful normal state on $\cM$. Let $\{E_j\}_{j=1}^N$ be a 
partition of unity in $ M_\omega$. Then, there exists a sequence of 
partitions of unity $\{e_1^{ \nu } ,e_2^{ \nu } ,\hdots,e_N^{ \nu } \}_{ \nu } $ 
in $\cM$ such that for  $j=1,2,\hdots,N$, $\{e_j^{ \nu } \}_{ \nu } $ is a 
representative sequence 
for $E_j$.
\end{lem}
\begin{proof}
{
For $1\leq k\leq N-1$, set $F_k:=\sum_{j=1}^k E_j$.  Thus
$0=F_0\leq F_1\leq\cdots\leq F_{N-1}\leq F_N=1$ is an increasing family of
projections in $M_\omega$.  We construct compatible projection lifts
inductively.  Suppose that $(q_{k-1}^\nu)_\nu$ is a projection lift of
$F_{k-1}$.  Using \cite[Lemma~1.1.5]{connes1975outer}, choose a projection lift
$(r_k^\nu)_\nu$ of $F_k$.  Since $F_{k-1}\leq F_k$, the positive contractions
\[
a_k^\nu:=(1-q_{k-1}^\nu)r_k^\nu(1-q_{k-1}^\nu)
\]
represent the projection $F_k-F_{k-1}$.  The projections
\[
s_k^\nu:=\chi_{[1/2,1]}(a_k^\nu)
\]
represent  $F_k-F_{k-1}$ as well.  Consequently,
$q_k^\nu:=q_{k-1}^\nu+s_k^\nu$ is a projection, $q_{k-1}^\nu\leq q_k^\nu$, and
$(q_k^\nu)_\nu$ represents $F_k$.  Starting with $q_0^\nu=0$, this yields
\[
0=q_0^\nu\leq q_1^\nu\leq\cdots\leq q_{N-1}^\nu\leq q_N^\nu=1
\]
for every $\nu$.  For $1\leq j<N$, set
\[
e_j^\nu:=q_j^\nu-q_{j-1}^\nu \quad \text{ and }
\quad e_N^\nu:=1-q_{N-1}^\nu.
\]
Then $\{e_1^\nu,\ldots,e_N^\nu\}$ is a partition of unity in $M$ for every
$\nu$, and $(e_j^\nu)_\nu$ represents $F_j-F_{j-1}=E_j$.
}
\end{proof}
\subsection{Free and centrally free actions}
Let $M$ be a von Neumann algebra with separable predual. An automorphism 
$\theta\in \aut( 
M)$ is called \emph{properly outer} if 
$\theta|_{P MP}$ is not inner for any nonzero invariant central 
projection $P$. We say that $\theta\in \aut( M)$ is \emph{strongly outer} if 
the restriction of $\theta$ to the relative commutant of any countable 
$\theta$-invariant subset of $ M$ is properly outer. Let us denote by 
$\mathcal{CT}( M)$ the collection of all centrally trivial 
automorphisms on $ M$:
\[\mathcal{CT}( M)=\left\{\theta\in \aut( M): 
\theta_{\omega}=\text{id}\in\aut\left( M_{\omega}\right)\right\}
\, .
\]
An automorphism $\theta\in \aut( M)$ is called \emph{properly 
centrally non-trivial} if for any non-zero invariant central projection $P\in  
M$, its 
restriction to $P MP$ is not
centrally trivial.

Let $\Gamma$ be a discrete group acting on a  von Neumann algebra 
$ M$ by automorphisms. The action $\alpha:\Gamma\to 
\aut( M)$ is said to be {\emph{free}} if $\alpha_g$ is properly outer for
any $g\ne e$. The action $\alpha$ is said to be {\emph{strongly free}} if  $\alpha_g$
is strongly outer for any $g\ne e$. We say that $\alpha$ is {\emph{centrally free}} if  $\alpha_g$ is
properly centrally non-trivial for any $g\ne e$. We recall the following.

\begin{lem} \cite[Lemma 5.6]{ocneanu2006actions}
\label[lemma]{lem:ocneanu_stronglyouter}
Let $\cM$ be a von Neumann algebra with separable predual. Let $\alpha$ be a 
properly 
centrally nontrivial automorphism of $\cM$. Then $\alpha_\omega$ is a strongly 
outer automorphism of $\cM_\omega$. 
\end{lem} 
\subsection{Crossed product $C^*$-algebras}
Let $\Gamma$ be a discrete group, let $ A $ be a $C^*$-algebra, and let 
$\alpha \colon \Gamma \to \aut(A)$ be an action. We denote the canonical 
conditional expectation by $\mathbb{E} \colon A \rtimes_{\alpha,r} \Gamma \to 
A$. (See \cite[Proposition~4.1.9]{BroOza08}.) We denote the canonical unitaries 
in $M ( A \rtimes_{\alpha,r} \Gamma )$ by $\{\lambda_s \}_{s \in \Gamma}$. They 
satisfy $\lambda_s a \lambda_s^* = \alpha_s(a)$ for any $s \in \Gamma$ and any 
$a$ in the canonical copy of $A$ in the reduced crossed product.

\subsection{Compressions} \label{subsec:compression}
For an inclusion $A\subseteq B$ such that $A$ contains an approximate
identity for $B$, we identify $A^{**}$ with its canonical unital von Neumann
subalgebra of $B^{**}$. If $P\in Z(B^{**})$ is a projection, the compression
map $\pi_P\colon A^{**}\to PA^{**}P$, given by $\pi_P(a)=PaP$, is a normal
surjective $*$-homomorphism. Its kernel is $(1-z_P)A^{**}$ for a central
projection $z_P\in A^{**}$, and its restriction to $z_PA^{**}$ is a normal
$*$-isomorphism onto $PA^{**}P$. Thus $PA^{**}P$ is isomorphic to a central
summand of $A^{**}$, although $P$ need not belong to $A^{**}$. Notice that $P z_P = P$. If $P$ is
invariant under an action preserving $A$, then $z_P$ is invariant and this
isomorphism is equivariant.

\section{von Neumann algebraic Rokhlin theorems}
\label{relativizedrokhlinconnes}
Connes' noncommutative Rokhlin theorem~\cite[Theorem~1.2.5]{connes1975outer} 
for a single automorphism is stated for automorphisms of finite von Neumann 
algebras with an invariant trace. 
We need a more general version that does not assume the trace is 
invariant. The proof of this stronger version is almost identical to the proof 
in  \cite{connes1975outer} (also see 
\cite[Chapter~XVII]{TakesakiIII}), except that the use of the ergodic theoretic 
Rokhlin lemma for measure preserving actions is replaced by the Rokhlin theorem 
for nonsingular actions (\cite{ornstein1980ergodic}). We nonetheless provide a
full proof for the reader's convenience. The Rokhlin theorem for essentially 
free measure class preserving actions of amenable groups was announced in 
\cite{ornstein1980ergodic}. There is a short note in 
\cite{ornstein1987entropy}, where the authors outline how to generalize the 
result from measure preserving systems to measure class preserving ones. 
Although we could not find a complete proof in the published literature, one is 
available in a Ph.D.~thesis by Jarrett \cite[Appendix~B]{jarrett2018non}.

\begin{prop}\label[proposition]{prop:connes}
Let $( N,\tau)$ be a finite von Neumann algebra, with $\tau$ a 
faithful normal tracial state. Assume that 
$\alpha:\mathbb{Z}\to\aut( N)$ is free. Then, for any positive 
integer $n$ and for any $\varepsilon>0$, there exists a partition of unity 
$\{E_j\}_{j=1}^n$ such that
\[\left\|\alpha_1(E_j)-E_{j+1}\right\|_\tau<\varepsilon,~j=1,2,\ldots,n,\]
where $E_{n+1}=E_1$.
\begin{proof}
Let $( N,\tau)$ be a finite von Neumann algebra. We can decompose 
$ N$ into an infinite direct sum $\overline{\bigoplus}_{k=0}^{\infty} 
 N_k$ where each summand is invariant, the action on the center 
$ Z( N_0)$ is free, the action on the center 
$ Z( N_1)$ is trivial, and for each $k>0$, the action on 
$ Z( N_k)$ is $k$-periodic and there is no non-zero invariant 
projection $P \in  Z( N_k)$ such that the restriction to $P 
 Z( N_k)$ has a shorter period. It suffices to prove the 
statement for each of those cases separately.

Let us first assume that $\alpha$ acts trivially on the center
$Z(N)$. Let $n$ be a fixed positive integer and $\varepsilon>0$ be given.
{Set $\eta:=\varepsilon^2/8$.} Denote by $\mathcal{U}(N)$ the
{unitary group of $N$}. Recall that $\|x\|_1=\tau(|x|)$.  Consider the collection 
$\mathcal{E}=\left\{\left(E_1,E_2,\ldots,E_n; U\right)\right\}$, where
\begin{enumerate}
    \item $E_1,E_2,\ldots,E_n$ are mutually orthogonal equivalent projections.
    \item $U\in \mathcal{U}( N)$ and $\|U-1\|_1\le {\eta} 
    \tau\left(\sum_{i=1}^nE_i\right)$.
    \item $U\alpha_1(E_i)U^*=E_{i+1}$, $E_{n+1}=E_1$.
\end{enumerate}
The set $\mathcal{E}$ is non-empty, since $(0,\ldots,0;1)\in\mathcal{E}$. We define an order on $\mathcal{E}$ by setting $\left(E_1,E_2,\ldots,E_n; U\right)\le \left(E_1',E_2',\ldots,E_n'; U'\right)$ if $E_i
\le E_i'$ for each $1\le i\le n$ and $\|U-U'\|_1\le 
{\eta}\tau\left(\sum_{i=1}^nE_i'-E_i \right)$. Let $\mathcal{F}$ be a 
totally 
ordered subset of $\mathcal{E}$. The map $\mathcal{O}:\mathcal{F}\to [0,1]$, 
defined by
\[\mathcal{O}\left(E_1,E_2,\ldots,E_n; U\right)=\tau\left(\sum_{i=1}^nE_i\right)\]
is order-preserving. Therefore, $\mathcal{F}$ contains a cofinal sequence 
$$y_m=\left(E_{1,m},E_{2,m},\ldots,E_{n,m}; U_m\right).$$
Using (2), we see that 
\begin{align*}\left\|U_{m+1}-U_m\right\|_1\le{\eta}\tau\left(\sum_{i=1}^nE_{i,m+1}-E_{i,m}\right),\end{align*}
and hence, $\sum_{m=1}^\infty\|U_{m+1}-U_m\|_1<\infty$. Since $\mathcal{U}( N)$ 
is complete with respect to the $\| \cdot \|_1$-norm, we see that $\lim_{m \to 
\infty} U_m=U$ 
exists and is in $\mathcal{U}( N)$. Moreover, $\{E_{i,m}\}_m$, an 
increasing sequence of projections, converges to a projection $E_i$ for each 
$i=1,2,\ldots,n$. We now observe that $\left(E_1,E_2,\ldots,E_n;U\right)$
satisfies (1), (2) and (3). {Thus every chain in $\mathcal{E}$ has an
upper bound in $\mathcal{E}$. By Zorn's lemma, $\mathcal{E}$ has a maximal
element, which we denote by $\left(E_1,E_2,\ldots,E_n;U\right)$.} We claim that $\sum_{i=1}^nE_i=1$. Suppose not. Set 
$P=1-\sum_{i=1}^nE_i\ne 0$. 
From (3), it follows that $U\alpha_1(P)U^*=P$. Therefore, 
$\tilde{\alpha}=\text{Ad}(U)\circ\alpha_1$ leaves $P NP$ invariant. Moreover, 
$\widetilde{\alpha}$ also acts freely on $PNP$.
Let $\tau_{P NP}(\cdot)=\frac{1}{\tau(P)}\tau(\cdot)$ and let
$\|.\|_{P NP}$ denote the $L^1$-norm associated with $\tau_{P NP}$.
{
Using \cite[Lemma~1.8, Chapter-XVII]{TakesakiIII}, we can find mutually
orthogonal, equivalent projections $\{P_1,P_2,\ldots,P_n\}$ in $PNP$ and a
unitary $\widetilde U\in\mathcal{U}(PNP)$ such that
\[
\|\widetilde U-P\|_{1,PNP}<\eta\,\tau_{PNP}\!\left(\sum_{i=1}^nP_i\right)
\quad\text{and}\quad
\widetilde U\widetilde\alpha(P_j)\widetilde U^*=P_{j+1}
\]
for $j=1,2,\ldots,n$, with $P_{n+1}=P_1$.  Let
$E_i':=E_i+P_i$ and define
\[
U':=(\widetilde U+1-P)U.
\]
Then $E_1',\ldots,E_n'$ are mutually orthogonal projections and
$U'\alpha_1(E_i')U'^*=E_{i+1}'$ for every $i$.  Furthermore,
\begin{align*}
\|\widetilde U+(1-P)-1\|_1
&=\tau(P)\|\widetilde U-P\|_{1,PNP}\\
&<\eta\tau(P)\tau_{PNP}\!\left(\sum_{i=1}^nP_i\right)\\
&=\eta\tau\!\left(\sum_{i=1}^nP_i\right).
\end{align*}
}
Therefore, we have
\begin{align*}
 \left\|U'-U\right\|_1&
 ={\left\|\left(\widetilde U+(1-P)-1\right)U\right\|_1}\\
 &{<\eta\tau\left(\sum_{i=1}^nP_i\right)}
\end{align*}
Thus,
\begin{align*}
 \left\|U'-1\right\|_1&\le \left\|U'-U\right\|_1+\left\|U-1\right\|_1\\&\le  
 {\eta}\tau\left(\sum_{i=1}^nP_i\right)+{\eta}\tau\left(\sum_{i=1}^nE_i\right)\\&={\eta}\tau\left(\sum_{i=1}^nE_i'\right)
   \, .
\end{align*}
Therefore, $\left(E_1',E_2',\ldots,E_n';U'\right)\in \mathcal{E}$ and in particular, $$\left(E_1,E_2,\ldots,E_n;U\right)<\left(E_1',E_2',\ldots,E_n';U'\right).$$
This contradicts the maximality of $\left(E_1,E_2,\ldots,E_n;U\right)$. Therefore, $\sum_{i=1}^nE_i=1$.
Finally, we observe that
\begin{align*}
\left\|\alpha_1(E_j)-E_{j+1}\right\|_2^2&
\le\left\|\alpha_1(E_j)-E_{j+1}\right\|\left\|\alpha_1(E_j)-E_{j+1}\right\|_1\\
&\le
 2\left\|\alpha_1(E_j)-E_{j+1}\right\|_1\\
 &=2\left\|\alpha_1(E_j)-U\alpha_1(E_{j})U^*\right\|_1\\
 &{\le4\|U-1\|_1\le4\eta=\varepsilon^2/2<\varepsilon^2},
 ~j=1,2,\ldots,n.   
\end{align*}
This completes the proof for the case in which $\alpha_1$ acts trivially on 
the center $ Z( N)$. 

 If $\alpha_1|_{ Z( N)}$ is essentially free, then we can use the Rokhlin theorem for {nonsingular} actions
(\cite[Theorem~B.2.4]{jarrett2018non}) to obtain the required projections 
 $\left\{E_1,E_2,\ldots,E_n\right\}$. It remains to consider the case in which 
 $\alpha_1$ is periodic with period $p\ge 1$.  Let $P$ be a projection in $Z(N)$ {such that} 
 $\left\{\alpha_1^j(P): 0\le j\le p-1\right\}$ is a partition of unity. It then 
 follows that $\alpha_1^p|_{ NP}$ is free and leaves the center $ Z\left( 
 NP\right)$ fixed. Therefore, using the arguments from before, we can find a 
 partition $\left\{F_1, F_2,\ldots, F_n\right\}$ of unity in $ NP$ such that
\[\left\|\alpha_1^p|_{ 
NP}(F_i)-F_{i+1}\right\|_2<\frac{\varepsilon}{p},~i=1,2,\ldots,n\text{
 and } F_{n+1}=F_1.\]
Let 
\[F_0=F_n\text{ and }G_{ip+j}=\alpha_1^j\left(F_i\right),~0\le j<p, 0\le i<n.\]
Observe that $\left\{G_k: 0\le k<np\right\}$ is a partition of unity 
in $ N$. Moreover, 
$\left\|\alpha_1(G_k)-G_{k+1}\right\|_2<\frac{\varepsilon}{p}$. Let
\[E_i=\sum_{k=0}^{p-1}G_{kn+i},~1\le i\le n,\text{ with the convention that } G_{pn}=G_0.\]
It then follows that $\left\{E_i: 1\le i\le n\right\}$ is a partition of unity 
in $ N$ and $\left\|\alpha_1(E_i)-E_{i+1}\right\|_2\le \varepsilon$ 
for 
each $i=1,2,\ldots,n$. As usual, $E_{n+1}=E_1$. This completes the proof. 
\end{proof} 
\end{prop}

 \cref{prop:connes} holds for finite von Neumann algebras{.
 However, the von Neumann algebras to which we will apply it arise as
 compressions $PA^{**}P$, where $P$ is a central projection in $B^{**}$,
 as in the statement of \cref{mainintmresult}, and these are
 typically not finite.} We bypass the finiteness assumption by working with the
 central sequence von Neumann algebras $ M_{\omega}$. This, in turn, gives us 
 projections in $ M_{\omega}$ which are Rokhlin towers. To obtain the 
 desired partition at the level of the von Neumann algebra $ M$ (which is not 
 necessarily finite), we use an index selection trick along with 
 \cref{lem:connes}.
\begin{lemma}
\label[lemma]{relConnesrelative}
Let $ N\subset M$ be an inclusion of von Neumann algebras with separable 
preduals such that $1_{ N}=1_{ M}$. Let $B\subset M$ be a weakly dense $C^*$-subalgebra. Suppose that $ N'\cap N^{\omega}\subseteq B'\cap N^{\omega}$ 
and suppose $\varphi$ is a faithful normal state on $ M$. Let 
$\alpha:\mathbb{Z}\to \aut( M)$ be an action such that
$\alpha_n(N)=N$ for every $n\in\mathbb{Z}$ and such that
$\alpha_\omega:\mathbb{Z}\to \aut( N_{\omega})$ is free. Then, given
${K}\in \mathbb{N}$, there exists a sequence
$\{e_{1}^{ \nu } ,e_{2}^{ \nu } ,\hdots,e_{{K}}^{ \nu } \}_{ \nu } $
of partitions of unity in $\cN$ such that
\begin{equation}
\label{equ:connes_equivariancevna}   
\left\| \alpha_{1}(e_j^{ \nu } )-e_{j+1}^{ \nu }  \right\|_\varphi^\# \underset{{
\nu } \to \omega}{\longrightarrow} 0 \,\,\forall j=1,2,\ldots,{K}.
\end{equation}
\begin{equation}
\label{equ:connes_commutevna} 
  \left\| e_j^{ \nu } b-be_j^{ \nu } \right\|_\varphi^\# \underset{{ \nu } \to
  \omega}{\longrightarrow} 0 \,\, \forall b\in \cB \text{ and }j\in\{1,2,\hdots
  ,{K}\}.
\end{equation} 
\begin{equation}
\label{equ:connes_commute2vna} 
  \left\| e_j^{ \nu } be_j^{ \nu } -be_j^{ \nu } \right\|_\varphi^\# \underset{{
  \nu } \to \omega}{\longrightarrow} 0 \,\, \forall b\in \cB\text{ and
  }j\in\{1,2,\hdots ,{K}\}.
\end{equation}
\end{lemma}
\begin{proof}
Applying \cref{prop:connes} to $\cN_\omega$, for ${K}\in\bN$, for any $n \in
\bN$, we obtain a partition of unity
$\{E_j^n\}_{j=1}^{{K}} \subseteq N_\omega\subseteq N'\cap N^\omega\subseteq B'\cap N^\omega$ such that
$\left\|\alpha_1(E_j^n)-E_{j+1}^n\right\|_{\varphi_\omega}<\frac{1}{n}$ for
$j=1,2,\hdots,{K}$, with $E_{{K}+1}^n=E_1^n$. Applying an index selection trick 
(e.g.~\cite[Section~5.5]{ocneanu2006actions}) to the elements 
$\{E_{1}^n,E_{2}^n,\hdots, E_{{K}}^n\}$ gives a  partition
of unity $\{E_j\}_{j=1}^{{K}} \subseteq \cN_\omega\subseteq N'\cap N^\omega$ such that
\[
\alpha_1(E_j)=E_{j+1} \text{ for } j={1,2,\hdots,{K}},\text{ with}\, E_{{K}+1}=E_1 \, .
\]
Finally, as a consequence of \cref{lem:connes}, and the fact that 
$\left\| \cdot \right\|_\varphi^\#$ defines the *-strong operator topology on 
bounded 
subsets of $\cM$, we obtain a sequence of partitions of unity $e_{1}^{ \nu }
,e_{2}^{ \nu } ,\hdots,e_{{K}}^{ \nu } $ in $\cN$, where $\{e_{j}^{ \nu } \}_{ \nu
} $ is a representative sequence for $E_{j}$ in $N'\cap N^\omega$, such that for all
$j\in\{1,2,\hdots,{K}\}$ we have
\[
\alpha_1(e_j^{ \nu } )-e_{j+1}^{ \nu }  \underset{{ \nu } 
\rightarrow\omega}{\longrightarrow} 0 \text{ *-strongly}, 
\] 
with $e_{{K}+1}^{ \nu } =e_1^{ \nu } $ for all ${ \nu } $. Equivalently,
\[
\left\| \alpha_1(e_j^{ \nu } )-e_{j+1}^{ \nu } \right\|_\varphi^\# \underset{{ \nu 
} 
	\rightarrow\omega}{\longrightarrow} 0 \, .
\]
This establishes equation~\eqref{equ:connes_equivariancevna}.
Equation~\eqref{equ:connes_commutevna} follows from the condition $N'\cap N^\omega\subseteq B'\cap N^\omega$. 
Consequently, using equation~\eqref{equ:connes_commutevna}, we see that
\begin{align*}
\left\|e_j^{ \nu } be_j^{ \nu } -e_j^{ \nu } b\right\|_\varphi 
&\leq\left\|be_j^{ \nu } -e_j^{ \nu } b\right\|_\varphi\underset{{ \nu } 
	\rightarrow\omega}{\longrightarrow} 0,
\end{align*}
and thus
\begin{align*}
\left\|e_j^{ \nu } be_j^{ \nu } -be_j^{ \nu } \right\|_\varphi &\leq 
\left\|e_j^{ \nu } be_j^{ \nu } -e_j^{ \nu } b\right\|_\varphi + \left\|e_j^{ 
\nu } b-be_j^{ \nu } \right\|_\varphi \\ & \leq 2\left\|be_j^{ \nu } -e_j^{ \nu 
} b\right\|_\varphi \underset{{ \nu } 
\rightarrow\omega}{\longrightarrow} 0. 
\end{align*}
{Applying the first of these estimates to $b^*$ gives the corresponding
convergence for the adjoint. Hence, for
$z_\nu=e_j^\nu b e_j^\nu-be_j^\nu$,
$\|z_\nu\|_{\varphi}^{\#\,2}=\|z_\nu\|_{\varphi}^2+
\|z_\nu^*\|_{\varphi}^2\longrightarrow0$.}
This establishes equation~\eqref{equ:connes_commute2vna}, thereby completing the proof.
\end{proof}

\begin{thm}\label[theorem]{thm:relativizedconnes}
	Let $ A \subseteq  B $ be separable $C^*$-algebras such that $A$ contains an approximate identity for $B$. 
	 Let $\omega$ be a free ultrafilter on 
	$\mathbb{N}$. Let 
	$\alpha:\bZ\to\aut(\cB)$ be an action which leaves $A$ invariant. Suppose 
	that for every central $\alpha^{**}$-invariant projection $P$ in $B^{**}$ 
	such that {$PB^{**}P$ has separable predual},  we have 
	\[
	(PA 
	P)'\cap (PA^{**}P)^\omega= (PB P)'\cap (PA^{**}P)^\omega,
	\]
	 and 
	$\alpha^{**}_\omega:\bZ\to\aut((PA^{**}P)_\omega)$ is free. Then, for any 
	$N\in\bN$ and for any  
	finite set of elements 
 	$S \subseteq \cB$,  there exists a net of mutually orthogonal, 
	contractive,  positive elements 
	$\{p_{1}^\beta,p_{2}^\beta,\hdots,p_{N}^\beta\}_{\beta}$ in $ A ${, where we set
	$p_{N+1}^\beta:=p_1^\beta$}, such that for all $j\in
	\{1,2,\hdots, N\}$ and for all $b \in S$ the following hold:  
	\begin{align}
		\alpha_1(p_j^\beta)-p_{j+1}^\beta & \to 0 \text{ *-strongly }.\\
		b p_j^\beta-p_j^\beta b  & \to 0\text{ *-strongly}.\\
		p_j^\beta b p_j^\beta-b  p_j^\beta &  \to 0\text{ *-strongly}.\\
		\sum_{j=1}^N p_j^\beta-1 & \to 0\text{ *-strongly }.\\
		(p_j^\beta)^2-p_j^\beta & \to 0\text{ *-strongly }.
	\end{align} 
\end{thm} 

\begin{proof}

We prove this in two steps.\\
Step 1: We claim that for any $N\in\bN$ and any finite set 
$S \subseteq \cB$ and for any $*$-strong neighborhood $V$ of $0$ 
in $B^{**}$, there 
exist mutually orthogonal projections $e_{1} 
,e_{2},\hdots 
e_{N}$ in $ A ^{**}$ such that, for any $j\in\{1,2,\hdots,N\}$ and for any
$b \in S$ the following hold:
\begin{enumerate}[label=(\Alph*)]
    \item\label{equ:connes_equivariancecstarstar} \[\alpha_1(e_j)-e_{j+1} \in V 
    \text{ with } e_{N+1} =e_1 
    . \]
    \item \label{equ:connes_commutecstarstar} 
  \[e_j b-be_j \in V
  . \]
  \item \[\label{equ:connes_commute2cstarstar} 
   e_j be_j-be_j \in V 
   . \]
   \item \[\label{equ:connes_pou}
   \sum_{j=1}^N e_j -1\in V 
   . \]
\end{enumerate} 

\cref{relConnesrelative} does not directly apply, as $A^{**}$ and
$B^{**}$ typically {do} not have separable preduals. This is corrected as 
follows. Pick a normal state $\varphi$ and $\varepsilon>0$ such that $\{b 
\in B^{**} \mid \|b\|_{\varphi}^{\#} < \varepsilon \} \subseteq V$. 

 Let $P$ be the smallest central $\alpha^{**}$-invariant 
projection which dominates the central support of $\varphi$. Concretely,
let $z_\varphi$ denote the central support of the state  $\varphi$. Then
$P=\bigvee_{n\in\mathbb Z}\alpha_n^{**}(z_\varphi)$. For each $n$, the GNS
representation of the restriction of $\varphi\circ\alpha_n^{**}$ to the
separable algebra $B$ acts on a separable Hilbert space, and the countable
direct sum of these representations is faithful on $PB^{**}P$. Thus
$PB^{**}P$ has separable predual. Its von Neumann subalgebra $PA^{**}P$ has
separable predual as well.
We may assume without loss
of generality that
$\varphi|_{P{\cB}^{**}P}$ is a
faithful normal state{. Indeed, since $P\cB^{**}P$ has a
separable predual, it admits a faithful normal state $\rho$, which we view as
a normal state on $\cB^{**}$ by setting $\rho(x):=\rho(PxP)$ for
$x\in\cB^{**}$; we then replace $\varphi$ by $\frac{1}{2}(\varphi+\rho)$ and
$\varepsilon$ by $\varepsilon/2$. The state $\frac{1}{2}(\varphi+\rho)$ is
normal, its central support is still dominated by $P$, and its restriction to
$P\cB^{**}P$ is faithful. Moreover, for any
$b\in\cB^{**}$, we have
$\|b\|_{\varphi}^{\#}\le\sqrt{2}\,\|b\|_{\frac{1}{2}(\varphi+\rho)}^{\#}$, so
that $\{b\in B^{**} \mid
\|b\|_{\frac{1}{2}(\varphi+\rho)}^{\#}<\varepsilon/2\}\subseteq\{b\in B^{**}
\mid \|b\|_{\varphi}^{\#}<\varepsilon\}\subseteq V$.}
Using the fact that 
$\alpha^{**}_\omega:\bZ\to\aut((P\cA^{**}P)_\omega)$ is   free and $PBP$ is weakly dense in $PB^{**}P$, by 
\cref{relConnesrelative}, for $N\in\bN$, we get a partition of unity 
$e_1,e_2,\hdots,e_N\in P\cA^{**}P$ such that 
\begin{equation}\label{equ:connes_one}
   \left\| \alpha_{1}(e_j)-e_{j+1} \right\|_\varphi^\# < \varepsilon 
   \,\,\forall j=1,2,\ldots,N. 
\end{equation}
\begin{equation}\label{equ:connes_two}
  \left\| e_j b-be_j \right\|_\varphi^\# < \varepsilon \,\, 
  \forall b\in S \text{ and }j\in\{1,2,\hdots ,N\}.
\end{equation} 
\begin{equation}\label{equ:connes_three}
 \left\| e_j be_j -be_j \right\|_\varphi^\# < \varepsilon  \,\, 
 \forall b\in S \text{ and }j\in\{1,2,\hdots ,N\}.
\end{equation}
\begin{equation}\label{equ:connes_four}
\left\|P-\sum_{j=1}^N e_j \right\|_\varphi^\#=0 .
\end{equation}
Recall the notation from \cref{subsec:compression}. Let $\widetilde e_j:=(\pi_P|_{z_PA^{**}})^{-1}(e_j)$ for $1\leq j\leq N$.
These are mutually orthogonal projections in $A^{**}$ with
$P\widetilde e_jP=e_j$. Since $P$ is central and invariant, and
$\varphi(x)=\varphi(PxP)$ for $x\in B^{**}$, using $Pz_P = P$, the preceding estimates hold
with $e_j$ replaced by $\widetilde e_j$ and $P$ in
\eqref{equ:connes_four} replaced by $1$. Relabeling these lifts as $e_j$,
we obtain conditions
\ref{equ:connes_equivariancecstarstar},\ref{equ:connes_commutecstarstar},
	\ref{equ:connes_commute2cstarstar} and \ref{equ:connes_pou} above.

Step 2: We now show that we can find a net of mutually orthogonal 
positive elements in $ 
A $ which satisfies the conclusion of \cref{thm:relativizedconnes}.

Let $\varphi$ be a normal state on $B^{**}$ and let $\varepsilon>0$. We need to 
show that we can find 
mutually orthogonal positive elements $p_1,p_2,\hdots,p_N$ in $A$, with $p_{N+1} = 
p_1$ such that for $j=1,2,\hdots,N$ and any $b \in S$ we have
	\begin{align}
	\left\| \alpha_1(p_j)-p_{j+1} \right\|_{\varphi}^{\#} < \varepsilon ,\\
	\left\| b p_j-p_j b \right \|_{\varphi}^{\#} < \varepsilon ,\\
	\left\| p_j b p_j-b  p_j \right\|_{\varphi}^{\#} < \varepsilon ,\\
	\left\| 1 - \sum_{j=1}^N p_j \right \|_{\varphi}^{\#} < \varepsilon , \\
	\left\| (p_j)^2-p_j \right\|_{\varphi}^{\#} < \varepsilon .
\end{align}

Denote by {$\sigma:\bZ\to\operatorname{Aut}(C(\bZ_N))$} the translation action, that is, 
$\sigma_l(f)(s)=f(s-l)$ for $s,l\in \bZ$. For $j\in\{1,2,\hdots,N\}$, let 
$f_j\in C(\bZ_N)$ be the projection in $C(\bZ_N)$ such that $f_j(j)=1$ and zero 
otherwise, with $f_{N+1}=f_1$.
Using Step 1, we can find {a net of order-zero maps
$\Theta^\nu:C(\bZ_N)\to A^{**}$ that are asymptotically equivariant; explicitly, they satisfy}
\begin{align*}
\alpha^{**}_1(\Theta^{ \nu } (f_j))&-\Theta^{ \nu } (\sigma_1(f_j))\to 0\text{ 
*-strongly },\,\text{and}\\
\Theta^{ \nu } (f_i)b&-b\Theta^{ \nu } (f_i)\to 0\text{ *-strongly, }\forall 
b\in  S. 
\end{align*}
Fix $\nu$ such that for all $b \in S$ and for $j=1,2,\ldots,N$ we have
	\begin{align}
	\| \alpha_1(\Theta^{ \nu } (f_j) )-\Theta^{ \nu } (f_{j+1}) 
	\|_{\varphi}^{\#} < {\varepsilon/2} ,\\
	\| b \Theta^{ \nu } (f_j ) -\Theta^{ \nu } (f_j ) b  \|_{\varphi}^{\#} < 
	{\varepsilon/2} ,\\
	\left\| \Theta^{ \nu } (f_j ) b \Theta^{ \nu } (f_j ) -b  \Theta^{ \nu } (f_j)  
	\right\|_{\varphi}^{\#} < {\varepsilon/2} ,\\
	\left\| 1 - \sum_{j=1}^N \Theta^{ \nu } (f_j ) \right \|_{\varphi}^{\#} < 
	{\varepsilon/2}
	. 
\end{align}

Using the Kaplansky density-type lemma
\cite[Lemma~1.1]{hirshberg2012decomposable}, choose a net
$(\Psi^\mu_\nu)_\mu$ of completely positive contractive order-zero maps
$C(\bZ_N)\to A$ such that
$\Psi^\mu_\nu(f)\to\Theta^\nu(f)$ $*$-strongly for every
$f\in C(\bZ_N)$. For $j\in\{1,2,\ldots,N\}$, let
\[
p_{j,\nu}^\mu:=\Psi^\mu_\nu(f_j),
\qquad p_{N+1,\nu}^\mu:=p_{1,\nu}^\mu.
\]
For every $\mu$, the elements $p_{1,\nu}^\mu,\ldots,p_{N,\nu}^\mu$ are
mutually orthogonal positive contractions. Let
$a_j:=\Theta^\nu(f_j)$ and $a_{N+1}:=a_1$. Since multiplication is jointly
$*$-strongly continuous on bounded sets and each $a_j$ is a projection, choose
$\mu$ such that, for every $b\in S$ and every $j\in\{1,\ldots,N\}$,
\begin{align*}
\left\|
\bigl(\alpha_1(p_{j,\nu}^\mu)-p_{j+1,\nu}^\mu\bigr)
-\bigl(\alpha_1(a_j)-a_{j+1}\bigr)
\right\|_\varphi^\#&<\frac{\varepsilon}{2},\\
\left\|
\bigl(bp_{j,\nu}^\mu-p_{j,\nu}^\mu b\bigr)
-\bigl(ba_j-a_jb\bigr)
\right\|_\varphi^\#&<\frac{\varepsilon}{2},\\
\left\|
\bigl(p_{j,\nu}^\mu bp_{j,\nu}^\mu-bp_{j,\nu}^\mu\bigr)
-\bigl(a_jba_j-ba_j\bigr)
\right\|_\varphi^\#&<\frac{\varepsilon}{2},\\
\left\|\sum_{j=1}^N(p_{j,\nu}^\mu-a_j)\right\|_\varphi^\#
&<\frac{\varepsilon}{2},\\
\left\|(p_{j,\nu}^\mu)^2-p_{j,\nu}^\mu\right\|_\varphi^\#
&<\varepsilon.
\end{align*}
The triangle inequality and the four inequalities used to choose $\nu$ give
\begin{align*}
\left\|\alpha_1(p_{j,\nu}^\mu)-p_{j+1,\nu}^\mu\right\|_\varphi^\#&<\varepsilon,\\
\left\|bp_{j,\nu}^\mu-p_{j,\nu}^\mu b\right\|_\varphi^\#&<\varepsilon,\\
\left\|p_{j,\nu}^\mu bp_{j,\nu}^\mu-bp_{j,\nu}^\mu\right\|_\varphi^\#&<\varepsilon,\\
\left\|1-\sum_{j=1}^N p_{j,\nu}^\mu\right\|_\varphi^\#&<\varepsilon,
\end{align*}
The last inequality in the first display gives
\[
\left\|(p_{j,\nu}^\mu)^2-p_{j,\nu}^\mu\right\|_\varphi^\#<\varepsilon.
\]
Since $\varphi$ and $\varepsilon$ were arbitrary, this proves Step~2.
\end{proof}
Ocneanu~\cite{ocneanu2006actions} proved a generalized Rokhlin-type theorem for 
centrally free actions of amenable groups. We show that if the group in 
question admits large monotiles, one can somewhat strengthen the conclusions of 
\cite[Theorem~6.1]{ocneanu2006actions}. In the paper, we only need the 
statement for finite cyclic groups, but the proof of the general statement is 
similar. The statement appears below in \cref{thm:relativizedocneanu}. 
\begin{prop}\label[proposition]{prtunityinvna}
Let $\Gamma$ be a countable amenable group.
Let $\cN\subset \cM$ be von Neumann algebras with separable preduals such 
that $1_\cN=1_\cM${, and let} $B\subset M$ be a weakly dense $C^*$-subalgebra. Let $\alpha:\Gamma\to\aut(\cM)$ be an action that leaves $N$ invariant,  and suppose that the restriction of $\alpha$ to $\cN$ is centrally
free. Suppose{, in addition, that} $N'\cap N^\omega\subseteq B'\cap N^\omega${, and let} $\varphi$ be a faithful normal
state on $\cM$ such that $\alpha|_{Z(\cN)}$ leaves $\varphi|_{Z(\cN)}$
invariant. If $T^{-1}$ {is} a monotile for $\Gamma$, then there exists a sequence $\{e_{t}^{\nu}\}_{t \in T, \nu \in \bN}$ of
partitions of unity in $\cN$ such that
\begin{equation}
\label{equ:equivariancevna}   
\left\| \alpha_{ts^{-1}}(e_s^{ \nu } )-e_t^{ \nu }  \right\|_\varphi^\#
\underset{{ \nu } \to \omega}{\longrightarrow} 0 \,\,\forall s,t\in T,
\end{equation}
\begin{equation}
\label{equ:commutevna} 
  \left\| e_t^{ \nu } b-be_t^{ \nu } \right\|_\varphi^\#\underset{{ \nu } \to 	
  \omega}{\longrightarrow} 0 \,\, 
  \forall b\in \cB \,\text{and}\,\, \forall t\in T ,
\end{equation} 
\begin{equation}
\label{equ:commute2vna} 
  \left\| e_t^{ \nu } be_t^{ \nu } -be_t^{ \nu } \right\|_\varphi^\#\underset{{ 
  \nu } \to 	
  	\omega}{\longrightarrow} 0  \,\, 
  \forall b\in \cB \,\text{and}\,\, \forall t \in T .
\end{equation}
\end{prop}
\begin{proof}
By \cref{prop:monotileableepsilon}, $T$ is an $\varepsilon$-paving set for 
any $\varepsilon$. Using {\cite[Theorem~6.1]{ocneanu2006actions}} along with
\cref{lem:ocneanu_stronglyouter} with $\varepsilon=\frac{1}{n}$, we get a 
sequence of partitions of unity $\{E_{t}^{(n)}\}_{t \in T}$ in 
$\cN_\omega\subset N'\cap N^\omega$ such that 
\[
{\left|\alpha_{ts^{-1}}(E_s^{(n)})-E_t^{(n)}\right|_{\varphi_\omega}}
<\frac{5|T|}{\sqrt{n}}\,\,
\forall s,t\in T,
\]
{Here $|x|_{\varphi_\omega}:=\varphi_\omega(|x|)$, following the notation
in \cite[Theorem~6.1]{ocneanu2006actions}. The same theorem also gives}
\[
[E_t^{(n)},\alpha_g(E_s^{(n)})]=0,\,\,\forall g\in\Gamma\text{ and 
}s,t \in T .
\] 
Apply the index selection trick (see 
\cite[Section~5.5]{ocneanu2006actions}) to the elements 
$\{E_{t}^{(n)}\}_{t \in T}$ obtained above to get a partition 
of unity $\{E_{t}\}_{t \in T}$ such that
\[
\alpha_{ts^{-1}}(E_s) = E_t \,\,
\forall s,t\in T,
\]
and
\[
{[E_t,\alpha_g(E_s)]=0},\,\,\forall g\in\Gamma\text{ and 
}s,t \in T .
\] 
Using \cref{lem:connes}, we obtain a sequence of 
partitions of unity $\{e_{t}^{\nu}\}_{t \in T}$ in $\cN$ such that  
$\{e_{t}^{\nu}\}_{\nu}$ is a representative sequence for $E_{t}$ in 
$N_\omega\subset N'\cap N^\omega$. Thus, for any $s,t\in T$, we have 
\[
\alpha_{ts^{-1}}(e_s^{ \nu } )-e_t^{ \nu } \underset{{ \nu } \to 
\omega}{\longrightarrow} 0 \text{ *-strongly}.
\]
Equivalently,
\[
\left\| \alpha_{ts^{-1}}(e_s^{ \nu } )-e_t^{ \nu } \right\|_\varphi^\# 
\underset{{ 
\nu } \to \omega}{\longrightarrow} 0 \, .
\]
This establishes equation~\eqref{equ:equivariancevna}.
Equation~\eqref{equ:commutevna} follows from the fact that $N'\cap N^\omega\subseteq B'\cap N^\omega$. 
Consequently, using equation~\eqref{equ:commutevna}, we see that
\begin{align*}
\left\|e_t^{ \nu } be_t^{ \nu } -e_t^{ \nu } b\right\|_\varphi 
&\leq\left\|be_t^{ \nu } -e_t^{ \nu } b\right\|_\varphi \underset{{ 
		\nu } \to \omega}{\longrightarrow}  0 .
\end{align*}
Thus,
\begin{align*}
\left\|e_t^{ \nu } be_t^{ \nu } -be_t^{ \nu } \right\|_\varphi &\leq 
\left\|e_t^{ \nu } be_t^{ \nu } -e_t^{ \nu } b\right\|_\varphi + \left\|e_t^{ 
\nu } b-be_t^{ \nu } \right\|_\varphi \\ & \leq 2\left\|be_t^{ \nu } -e_t^{ \nu 
} b\right\|_\varphi \underset{{ 
	\nu } \to \omega}{\longrightarrow}  0. 
\end{align*}
{The same calculation applied to $b^*$ gives convergence of the adjoints.
Thus, for $z_\nu=e_t^\nu b e_t^\nu-be_t^\nu$,
$\|z_\nu\|_{\varphi}^{\#\,2}=\|z_\nu\|_{\varphi}^2+
\|z_\nu^*\|_{\varphi}^2\longrightarrow0$.}
This establishes equation~\eqref{equ:commute2vna}, thereby completing the proof.
\end{proof}

\begin{thm}\label[theorem]{thm:relativizedocneanu}
	Let $ A \subseteq  B $ be separable $C^*$-algebras such that $A$ contains an approximate identity for $B$.
	Let $\omega$ be a free ultrafilter on $\bN$.  
	Let 
	$\alpha:\bZ_N\to\aut(\cB)$ be an action which leaves $A$ invariant. 
	 Suppose for every central $\alpha^{**}$-invariant projection $P$ in $B^{**}$ such that {$PB^{**}P$ has separable predual},  we have 
	  \[ (PA 
	P)'\cap (PA^{**}P)^\omega=  (PB P)'\cap (PA^{**}P)^\omega .
	\]
	 Suppose furthermore that 
	 $\alpha^{**}_\omega:\bZ_N\to\aut((PA^{**}P)_\omega)$ is strongly free. 
	 Then, for any finite set of 
	elements 
	$S \subseteq \cB$, there exists a net of mutually orthogonal, 
	positive contractions $\{p_{1}^\beta,p_{2}^\beta,\hdots,p_{N}^\beta\}_{\beta}$ 
	in $
	A ${, where we set $p_{N+1}^\beta:=p_1^\beta$}, such that the following hold  for
	all
	$j\in \{1,2,\hdots, N\}$ and for all $b \in S$:
	\begin{align}
		\alpha_1(p_j^\beta)-p_{j+1}^\beta &\to 0\text{ *-strongly },\\
		bp_j^\beta-p_j^\beta b &\to 0\text{ *-strongly},\\
		p_j^\beta b p_j^\beta-b p_j^\beta &\to 0\text{ *-strongly},\\
		\sum_{j=1}^N p_j^\beta-1 &\to 0\text{ *-strongly },\\
		(p_j^\beta)^2-p_j^\beta &\to 0\text{ *-strongly }.
	\end{align} 
\end{thm}

\begin{proof}
As in the proof of \cref{thm:relativizedconnes}, given a normal state 
$\varphi$ on $B^{**}$ and $\varepsilon>0$, one needs to find mutually 
orthogonal 
positive elements $p_{1},p_{2},\hdots,p_{N} \in A$ such that for 
$j=1,2,\ldots,N$ and for all $b \in S$ we have
	\begin{align}
	\| \alpha_1(p_j)-p_{j+1} \|^{\#}_{\varphi} < \varepsilon ,\\
	\| bp_j-p_j b \|^{\#}_{\varphi} < \varepsilon ,\\
	\| p_j b p_j-b p_j\|^{\#}_{\varphi} < \varepsilon ,\\
	\left\| \sum_{j=1}^N p_j-1\right\|^{\#}_{\varphi} < \varepsilon ,\\
	\| (p_j)^2-p_j\|^{\#}_{\varphi} < \varepsilon .
\end{align} 
{
Let $G=\mathbb Z_N$, and set
$\overline\varphi:=|G|^{-1}\sum_{g\in G}\varphi\circ\alpha_g$. Then
$\varphi\leq |G|\overline\varphi$, and hence
$\|x\|_\varphi^\#\leq\sqrt{|G|}\,\|x\|_{\overline\varphi}^\#$.
We therefore carry out the construction for the invariant state
$\overline\varphi$ with tolerance $\varepsilon/\sqrt{|G|}$. The rest of the
proof proceeds as in the proof of \cref{thm:relativizedconnes}. More precisely,
we repeat the proof of \cref{prtunityinvna} from the point where Ocneanu's
theorem is applied: central freeness in that proposition was used, via
\cref{lem:ocneanu_stronglyouter}, only to obtain strong freeness on the central
sequence algebra, which is assumed directly here. Finally, the entire finite
group $G$ is a monotile.
}
\end{proof}

\section{Structure of Intermediate Algebras}
\label{maintheoremproof}

In this section we prove Theorem \ref{mainintmresult}. We start with a
preparatory lemma. For a positive functional $\varphi$ on a $C^*$-algebra and
an element $z$ of that algebra, we write
$\|z\|_\varphi:=\varphi(z^*z)^{1/2}$.

\begin{lem}
\label[lemma]{correctiveepsilon}
Let $ A \subseteq  B $ be separable $C^*$-algebras such that $A$ contains an approximate identity for $B$. 
Let 
$\omega$ be a free ultrafilter on $\bN$. Let 
$\Gamma$ be a discrete countable group, 
and let $\alpha:\G\to\aut(B)$ be an action which leaves $A$ invariant. Suppose that for every central $\alpha^{**}$-invariant projection $P$ in $B^{**}$ such that {$PB^{**}P$ has separable predual},  we have 
\[
 (PA 
	P)'\cap (PA^{**}P)^\omega= (PB P)'\cap (PA^{**}P)^\omega,
	\]
	 and
$\alpha^{**}:\G\to\aut(PA^{**}P)$ is centrally free. Then, for any $s\in 
\Gamma\setminus\{e\}$,  
there exists $N \in\bN$ (depending only on $s$), such that for every 
finite set $F \subseteq B$, for any finite set of states
$\varphi_1,\varphi_2,\ldots,\varphi_M$ on
$B\rtimes_{\alpha,r}\Gamma$, and for any 
$\delta>0$, there exist mutually orthogonal, positive contractions $p_1, p_2, 
\ldots, p_N \in  A $ such that the following hold: 
 
\begin{enumerate}[label=(\alph*)]
\item \label{firstineq}\[\left\|\sum_{l=1}^Np_lx \lambda_s 
p_l\right\|_{\varphi_j}< \delta,~\text{for }x \in F\,\,\,\text{and } j=1,2,\ldots,M.\]
\item\label{secondineq} \[\left\|\sum_{l=1}^Np_lx 
p_l-x\right\|_{\varphi_j}<\delta,\text{ for }x \in F\,\,\,\text{and } j=1,2,\ldots,M.\]
\item\label{thirdineq} {$\|xp_l-p_lx\|_{\varphi_j}<\delta$ for every
$x\in F$, $j\in\{1,\ldots,M\}$, and $l\in\{1,\ldots,N\}$.}
\end{enumerate}
\begin{proof}
Let $\Gamma_s$ be the subgroup of $\G$ generated by $s\neq e\in\G$.
We first verify that the hypotheses of \cref{thm:relativizedconnes} and
\cref{thm:relativizedocneanu} pass to this cyclic subaction. Let $R$ be a
central $\alpha_s^{**}$-invariant projection in $B^{**}$ such that
$RB^{**}R$ has separable predual, and set
\[
Q:=\bigvee_{g\in\Gamma}\alpha_g^{**}(R).
\]
Then $Q$ is central and $\Gamma$-invariant. Since $\Gamma$ is countable,
$QB^{**}Q$ has separable predual. Indeed, write the elements of $\Gamma$ as a
finite or infinite sequence $(g_n)$ and successively remove the overlaps from
the projections $\alpha_{g_n}^{**}(R)$. This gives a finite or countable
orthogonal family of central projections with join $Q$, each of whose corners
is a central corner of a von Neumann algebra isomorphic to $RB^{**}R$.

The compression map $QA^{**}Q\to RA^{**}R$, given by $a\mapsto RaR$,
is a normal surjective $*$-homomorphism. As explained in
Section~\ref{Sec: Prelim}, it identifies $RA^{**}R$ equivariantly with an
$\alpha_s^{**}$-invariant central summand of $QA^{**}Q$. The inverse of
its restriction to this summand, followed by inclusion, induces an embedding
$(RA^{**}R)^\omega\longrightarrow(QA^{**}Q)^\omega$.

Under this embedding, every element of
$(RAR)'\cap(RA^{**}R)^\omega$ belongs to
$(QAQ)'\cap(QA^{**}Q)^\omega$, and hence, by the hypothesis for $Q$, commutes
with $QBQ$. Compressing by $R$, it therefore commutes with $RBR$. The reverse inclusion is
automatic, so
$(RAR)'\cap(RA^{**}R)^\omega=(RBR)'\cap(RA^{**}R)^\omega$.

Moreover, central freeness of the $\Gamma_s$-action on $QA^{**}Q$ restricts
to this invariant central summand, and hence passes to $RA^{**}R$. Thus, by
\cref{lem:ocneanu_stronglyouter}, the induced action on the central sequence
algebra is strongly free, and in particular free.

Now $\G_s\cong\bZ$ or $\bZ_N$ if the order of $s$ is $N$. If $s$ has
infinite order, set $N=2$. By \cref{thm:relativizedocneanu} if
$\G_s\cong\bZ_N$, or by \cref{thm:relativizedconnes} if
$\G_s\cong\bZ$, there exists a net of mutually orthogonal, positive
contractions $\{p_1^\beta, 
p_2^\beta, \ldots, p_N^\beta\}_\beta$ in $\cA$, such that for 
all $l\in\{1,2,\hdots,N\}$ and for all $x \in F$, we have 
\begin{align*}
\alpha_{s}(p_l^\beta)-p_{l+1}^\beta &\to 0\text{ *-strongly, with } 
p_{N+1}^\beta=p_1^\beta,\\
xp_l^\beta-p_l^\beta x &\to 0\text{ *-strongly},\\
p_l^\beta x p_l^\beta-x p_l^\beta &\to 0\text{ *-strongly},\\
\sum_{l=1}^N p_l^\beta-1 &\to 0\text{ *-strongly },\\
(p_l^\beta)^2-p_l^\beta &\to 0\text{ *-strongly }.
\end{align*} 
For each $j\in\{1,\ldots,M\}$, let
\[
\rho_j:=\varphi_j|_B\in B^*_+.
\]
We identify $\rho_j$ with its canonical normal extension to $B^{**}$. For the
fixed element $s\in\Gamma$ and for each $y\in B$, let
\[
\rho_j^s(a):=\rho_j\bigl(\alpha_{s^{-1}}^{**}(a)\bigr)
\quad\text{and}\quad
\rho_j^y(a):=\rho_j(y^*ay),
\qquad a\in B^{**}.
\]
These are normal positive functionals on $B^{**}$.

Recall that, on bounded subsets of $B^{**}$, the $*$-strong topology is
generated by the seminorms
\[
a\longmapsto
\|a\|_\rho^\#
:=
\bigl(\rho(a^*a)+\rho(aa^*)\bigr)^{1/2},
\qquad \rho\in(B^{**})_*^+.
\]
Thus, if a bounded net $(a_\beta)_\beta$ converges to zero $*$-strongly in
$B^{**}$, then
\[
\|a_\beta\|_\rho=\rho(a_\beta^*a_\beta)^{1/2}\longrightarrow0
\]
for every normal positive functional $\rho$ on $B^{**}$. Applying this fact to the equivariance, commutator, approximate-idempotence,
and partition-of-unity nets displayed above gives, for every
$j\in\{1,\ldots,M\}$, every $l\in\{1,\ldots,N\}$, and every $x\in F$,
\begin{align*}
\left\|\alpha_s(p_l^\beta)-p_{l+1}^\beta\right\|_{\rho_j^s}
&\longrightarrow0,\\
\left\|xp_{l+1}^\beta-p_{l+1}^\beta x\right\|_{\rho_j^s}
&\longrightarrow0,\\
\left\|xp_l^\beta-p_l^\beta x\right\|_{\rho_j}
&\longrightarrow0,\\
\left\|(p_l^\beta)^2-p_l^\beta\right\|_{\rho_j^x}
&\longrightarrow0,\\
\left\|\sum_{l=1}^N p_l^\beta-1\right\|_{\rho_j^x}
&\longrightarrow0.
\end{align*}

For $a,y\in B$, these seminorms are related to the seminorms induced by the
states $\varphi_j$ on the crossed product by
\begin{align*}
\|a\lambda_s\|_{\varphi_j}^2
&=\varphi_j(\lambda_s^*a^*a\lambda_s)\\
&=\rho_j\bigl(\alpha_{s^{-1}}(a^*a)\bigr)
=\|a\|_{\rho_j^s}^2,
\end{align*}
and
\begin{align*}
\|ay\|_{\varphi_j}^2
&=\varphi_j(y^*a^*ay)
=\rho_j^y(a^*a)
=\|a\|_{\rho_j^y}^2.
\end{align*}
Let
\[
L:=\max\bigl(\{1\}\cup\{\|x\|:x\in F\}\bigr).
\]
Then $\|x\|\leq L$ for every $x\in F$.

Given $\delta>0$,
choose $\beta_0$ such that for all $\beta \geq \beta_0$,  for  
$l = 1,2,\ldots,N$, for $j= 1,2,\ldots M$ and for all $x \in F$ we have 
\begin{align*}
\left\|\alpha_s(p_l^\beta)-p_{l+1}^\beta\right\|_{\rho_j^s}<
\frac{\delta}{2NL}, \\
\left\|x p_{l+1}^\beta-p_{l+1}^\beta 
x \right\|_{\rho_j^s}<\frac{\delta}{2N}, \\
\left\|x p_l^\beta-p_l^\beta x\right\|_{\rho_j} <\frac{\delta}{3N}, \\
\left\|(p_l^\beta)^2-p_l^\beta\right\|_{\rho_j^{x}}<\frac{\delta}{3N}, \\
\left\|\sum_{l=1}^N p_l^\beta-1\right\|_{\rho_j^{x}}<\frac{\delta}{3N}.
\end{align*}
Let $\beta \geq \beta_0$. For any $x\in F$ and for any $j = 1,2,\ldots M$, we 
have
\begin{align*}
\left\|\sum_{l=1}^Np_l^\beta x\lambda_sp_l^\beta\right\|_{\varphi_j} &\leq 
\sum_{l=1}^N\left\|p_l^\beta x\lambda_sp_l^\beta\right\|_{\varphi_j}\\ &\leq 
\sum_{l=1}^N\left\|p_l^\beta x\alpha_s ( p_l ^\beta ) \lambda_s-p_l^\beta x 
p_{l+1}^\beta\lambda_s\right\|_{\varphi_j}\\ & + \sum_{l=1}^N\left\|p_l^\beta 
x p_{l+1}^\beta\lambda_s-p_l^\beta p_{l+1}^\beta x 
\lambda_s\right\|_{\varphi_j}\\ & \leq 
\sum_{l=1}^N\left\|p_l^\beta\right\|\left\|x\right\| 
\left\|\left(\alpha_s ( p_l^\beta ) 
-p_{l+1}^\beta\right)\lambda_s\right\|_{\varphi_j}\\
 &  + 
\sum_{l=1}^N\left\|p_l^\beta\right\|\left\|\left(xp_{l+1}^\beta-p_{l+1}^\beta 
x\right)\lambda_s\right\|_{\varphi_j} .
\end{align*}
Note that
\begin{align*}\left\|\left(\alpha_s ( p_l^\beta ) 
-p_{l+1}^\beta\right)\lambda_s\right\|_{\varphi_j}^2&
	=\varphi_j\left(\lambda_s^*\left(\alpha_s (p_l^\beta ) 
	-p_{l+1}^\beta\right)^*
	\left(\alpha_s ( p_l^\beta ) 
	-p_{l+1}^\beta\right)\lambda_s\right)\\&=\rho_j^s
	\left(\left(\alpha_s ( p_l^\beta ) -p_{l+1}^\beta\right)^*\left(\alpha_s ( 
	p_l^\beta ) -p_{l+1}^\beta\right)\right)\\
	&=\left\|\alpha_s ( p_l^\beta ) 
	-p_{l+1}^\beta\right\|_{\rho_j^s}^2 .
	\end{align*}
Similarly, we obtain  
\[\left\|\left(x p_{l+1}^\beta-p_{l+1}^\beta 
x\right)\lambda_s\right\|_{\varphi_j}=\left\|xp_{l+1}^\beta-p_{l+1}^\beta 
x\right\|_{\rho_j^s} .
\]
Therefore, it follows that
\begin{align*}
\left\|\sum_{l=1}^Np_l^\beta x\lambda_sp_l^\beta\right\|_{\varphi_j}& \le 
\sum_{l=1}^N\|x\|\left\|\alpha_s ( p_l^\beta ) 
-p_{l+1}^\beta\right\|_{\rho_j^s}+\sum_{l=1}^N
\left\|xp_{l+1}^\beta-p_{l+1}^\beta x\right\|_{\rho_j^s}\\ &<
\frac{\delta}{2}+\frac{\delta}{2}=\delta .
\end{align*}
Likewise, for $\beta\geq\beta_0$, for any $x \in F$ and for $j=1,2,\ldots, M$, we 
have
\begin{align*}
\left\|\sum_{l=1}^N p_l^\beta xp_l^\beta-x\right\|_{\varphi_j} &\leq 
\left\|\sum_{l=1}^N p_l^\beta x p_l^\beta-\sum_{l=1}^N(p_l^\beta)^2 
x\right\|_{\varphi_j} + 
\left\|\sum_{l=1}^N\left((p_l^\beta)^2-p_l^\beta)\right)x\right\|_{\varphi_j} 
\\ & + \left\|\left(\sum_{l=1}^N p_l^\beta-1\right)x\right\|_{\varphi_j}\\ 
&\leq \sum_{l=1}^N\left\|p_l^\beta\right\|\left\| x p_l^\beta-p_l^\beta 
x\right\|_{\varphi_j} 
+\sum_{l=1}^N\left\|\left((p_l^\beta)^2-p_l^\beta)\right)x\right\|_{\varphi_j} 
\\ & +
\left\|\left(\sum_{l=1}^N 
p_l^\beta-1\right)x\right\|_{\varphi_j}\\&\le\sum_{l=1}^N\|x 
p_l^\beta-p_l^\beta 
x\|_{\rho_j}+\sum_{l=1}^N\left\|\left((p_l^\beta)^2-p_l^\beta\right)\right\|_{\rho_j^{x}}
 \\ & +
\left\|\left(\sum_{l=1}^N p_l^\beta-1\right)\right\|_{\rho_j^{x}}\\ & <\delta .
\end{align*}
Thus, we may take $\{p_1,p_2,\hdots,p_N\}$ to be
$\{p_1^\beta,p_2^\beta,\hdots,p_N^\beta\}$ for the chosen $\beta$.
Since $\rho_j=\varphi_j|_B$, the third inequality in the choice of $\beta_0$ gives~\ref{thirdineq}.
\end{proof}
\end{lem}
\begin{remark}
\label[remark]{normrelation}
Let $\varphi$ be a state on $B\rtimes_{\alpha,r}\Gamma$, let
$\rho:=\varphi|_B$, and, for $s\in\Gamma$, let
$\rho^s:=\rho\circ\alpha_{s^{-1}}$. Then, for every $x\in B$,
\begin{align*}
\|\alpha_s(x)\|_{\rho^s}^2
&=\rho\bigl(\alpha_{s^{-1}}(\alpha_s(x)^*\alpha_s(x))\bigr)\\
&=\rho(x^*x)
=\|x\|_\rho^2.
\end{align*}
\end{remark}
\begin{proof}[Proof of \cref{mainintmresult}]
	Let $D:=B\rtimes_{\alpha,r}\Gamma$. Let $ C $ be an intermediate $C^*$-algebra of the form 
	\[
	 A \rtimes_{\alpha,r} \Gamma\subseteq C \subseteq B \rtimes_{\alpha,r} 
	\Gamma .
	\]
	 We want to show that 
	$\mathbb{E}( C )\subset C $. 
	 
	Suppose for every central $\alpha^{**}$-invariant projection $P\in B^{**}$ such that {$PB^{**}P$ has separable predual}, we have 
	\[
	(PAP)'\cap (PA^{**}P)^{\omega}=(PBP)'\cap (P A^{**}P)^{\omega},
	\]
	and $\alpha^{**}:\Gamma\to\aut(PA^{**}P)$ is centrally free.

Fix $x\in C$. We first prove that for every state $\varphi$ on $D$ and every
$\delta>0$, there exists $x_{\delta,\varphi}\in C$ such that
$\|x_{\delta,\varphi}-\mathbb{E}(x)\|_\varphi<\delta$. To this end, fix a
state $\varphi$ on $D$ and $\delta>0$. Since finite Fourier sums are norm
dense in $D$ and $\mathbb E$ is contractive, pick

$b_1,b_2,\hdots,b_m\in  B $, and $s_1,s_2,\hdots,s_m\in \G\setminus\{e\}$ 
such that
\begin{equation}
\label{equ:norminequlity}
\left\|x-\sum_{i=1}^m b_i\lambda_{s_i}-\mathbb{E}(x)\right\|<\frac{\delta}{3}.
\end{equation}

For $i=1,2,\ldots,m$, let $N_i$ be the order of $s_i$ if it is finite, and let 
$N_i:=2$ if
it is infinite. If $m=0$, then $x_{\delta,\varphi}:=x$ satisfies
$\|x_{\delta,\varphi}-\mathbb E(x)\|_\varphi<\delta/3$. Thus assume
$m\geq1$. We show that for any
$\varepsilon>0$ and $j\in\{1,2,\hdots,m\}$, there exist contractions
$q_1,q_2,\ldots,q_Q\in A$ such that
$\sum_{l=1}^Qq_lq_l^*\leq1$,

\begin{align}
\label{equ:inductionhypothesisone}
\left\|\sum_{l=1}^{Q} q_l\left(\sum_{i=1}^j 
b_i\lambda_{s_i} \right)q_l^*\right\|_\varphi &\leq 
\sum_{i=1}^{j}\left\|\sum_{l=1}^{Q} q_l b_i\lambda_{s_i} q_l^*\right\|_\varphi 
<\varepsilon, 
\end{align}
and
\begin{align}
\label{equ:inductionhypothesistwo}
\left\|\sum_{l=1}^{Q} q_l \mathbb{E}(x) q_l^*-\mathbb{E}(x)\right\|_\varphi 
&<\varepsilon \, .
\end{align}

We prove this by induction on $j$. For $j=1$, apply
\cref{correctiveepsilon} with the finite set $\{b_1,\mathbb E(x)\}$, the state
$\varphi$, and tolerance $\varepsilon$. Since the resulting $p_l$ are mutually orthogonal positive contractions,
$\sum_l p_lp_l^*=\sum_l p_l^2\leq1$. Thus, taking $q_l:=p_l$ also gives the
inequality $\sum_l q_lq_l^*\leq1$ in the inductive assertion.

Assume that the inductive assertion holds for $j-1$. Choose
 contractions 
$q_1,q_2,\ldots,q_Q \in\cA$ such that  
equations~\eqref{equ:inductionhypothesisone} and 
\eqref{equ:inductionhypothesistwo} hold for $j-1$ in place of $j$, with the following parameters:
\begin{align}\label{equ:firstmainthmind}
\left\|\sum_{l=1}^Q q_l\left(\sum_{i=1}^{j-1} b_i\lambda_{s_i} 
\right)q_l^*\right\|_\varphi & \leq\sum_{i=1}^{j-1}\left\|\sum_{l=1}^Q q_l 
b_i\lambda_{s_i} q_l^*\right\|_\varphi<\frac{\varepsilon}{12N_j },
\end{align}
\begin{align}\label{equ:firstmainthmparttwoind}
\left\|\sum_{l=1}^Q q_l\mathbb{E}(x)q_l^*-\mathbb{E}(x)\right\|_\varphi 
&<\frac{\varepsilon}{6} .
\end{align}
For $i=1,2,\hdots,j$, define $x_i:=\sum_{l=1}^Q q_l b_i\alpha_{s_i}(q_l^*)$.
{Also set $y:=\sum_{l=1}^Qq_l\mathbb E(x)q_l^*$.}

Let $\rho:=\varphi|_B$. For $i=1,2,\ldots,j$, define states $\varphi_i$ on
$D$ by
$\varphi_i:=\varphi\circ\operatorname{Ad}(\lambda_{s_i}^*)$. Then
$\varphi_i|_B=\rho\circ\alpha_{s_i^{-1}}=\rho^{s_i}$. Since
$j\geq2$, apply \cref{correctiveepsilon} with $s=s_j$, with the finite set
$\{x_j,y\}\cup\{\alpha_{s_i}^{-1}(x_i)\mid i=1,2,\ldots,j-1\}$, and
the single state $\varphi$ to choose mutually 
orthogonal 
positive contractions $\{p_1,p_2,\hdots,p_{N_j}\}$ such that:
\begin{equation}\label{equ:secmainthmind}
\left\|\sum_{k=1}^{N_j}p_k x_j\lambda_{s_j} 
p_k\right\|_\varphi<\frac{\varepsilon}{6},
\end{equation}
\begin{equation}\label{equ:secmainthm2ind}
\left\|\sum_{k=1}^{N_j}p_k\left(\sum_{l=1}^Q 
q_l\mathbb{E}(x)q_l^*\right)p_k-\sum_{l=1}^Q 
q_l\mathbb{E}(x)q_l^*\right\|_{\varphi}<\frac{\varepsilon}{6},
\end{equation}
and for $i=1,2,\ldots,j-1$ and for $k=1,2,\ldots,N_j$ we have
\begin{equation}\label{equ:mainthmcommutation}
\left\|\alpha_{s_i}^{-1}(x_i)p_k-p_k 
\alpha_{s_i}^{-1}(x_i)\right\|_{\varphi}<\frac{\varepsilon}{12N_j(j-1)}. 
\end{equation}
By Remark~\ref{normrelation}, equation (\ref{equ:mainthmcommutation}) implies 
that for $i=1,2,\ldots,j-1$ and for 
 $k=1,2,\ldots,N_j$ we have
\begin{equation}\label{equ:firstmainthm2ind}
\left\|x_i\alpha_{s_i}(p_k)-\alpha_{s_i}(p_k) 
x_i\right\|_{\varphi_{i}}<\frac{\varepsilon}{12N_j(j-1)}.
\end{equation}

For $L:=l+(k-1)Q$ with $l\in\{1,2,\hdots,Q\}$ and
$k\in\{1,2,\hdots,N_j\}$, set $r_L:=p_kq_l\in A$. Moreover,
\[
\sum_{L=1}^{N_jQ}r_Lr_L^*
=\sum_{k=1}^{N_j}p_k\left(\sum_{l=1}^Qq_lq_l^*\right)p_k
\leq\sum_{k=1}^{N_j}p_k^2\leq1.
\]
Observe that

\begin{align*}
\left\|\sum_{L=1}^{N_jQ}r_L 
\left(\sum_{i=1}^{j-1}b_i\lambda_{s_i}\right){r_L}^* \right\|_\varphi= 
&\left\|\sum_{k=1}^{N_j}\sum_{l=1}^Qp_k q_l 
\left(\sum_{i=1}^{j-1}b_i\lambda_{s_i}\right)q_l^* p_k \right\|_\varphi\\&\leq 
\sum_{i=1}^{j-1}\sum_{k=1}^{N_j}\left\|\sum_{l=1}^Q q_l b_i\lambda_{s_i}q_l^* 
p_k\right\|_\varphi\\ &=
\sum_{i=1}^{j-1}\sum_{k=1}^{N_j}\left\|\sum_{l=1}^Q q_l b_i\alpha_{s_i}(q_l^*) 
\alpha
_{s_i}(p_k) \right\|_{\varphi_i}\\&=
\sum_{i=1}^{j-1}\sum_{k=1}^{N_j}\left\|x_i \alpha
_{s_i}(p_k) \right\|_{\varphi_i}\\ & \leq
\sum_{i=1}^{j-1}\sum_{k=1}^{N_j}\left\|x_i\alpha_{s_i}(p_k)-\alpha_{s_i}(p_k) 
x_i\right\|_{\varphi_i}+
\sum_{i=1}^{j-1}\sum_{k=1}^{N_j}\left\|\alpha_{s_i}(p_k) 
x_i\right\|_{\varphi_{i}}\\
&\le\sum_{i=1}^{j-1}\sum_{k=1}^{N_j}\left\|x_i\alpha_{s_i}(p_k)-\alpha_{s_i}(p_k)
 x_i\right\|_{\varphi_i}+\sum_{i=1}^{j-1}\sum_{k=1}^{N_j}\|\alpha_{s_i}(p_k)\|
\|x_i\|_{\varphi_i}\\
&\le\sum_{i=1}^{j-1}\sum_{k=1}^{N_j}
\left\|x_i\alpha_{s_i}(p_k)-\alpha_{s_i}(p_k) 
x_i\right\|_{\varphi_i}+N_j\sum_{i=1}^{j-1}\left\|\sum_{l=1}^Q q_l 
b_i\lambda_{s_i}q_l^*\right\|_\varphi \\ 
&\stackrel{(\ref{equ:firstmainthm2ind}),(\ref{equ:firstmainthmind})}{<}
\frac{\varepsilon}{12}+\frac{\varepsilon}{12}=\frac{\varepsilon}{6}.
\end{align*}
The same calculation, applied to each Fourier term separately before
summing over $i$, gives
\begin{align*}
	\sum_{i=1}^{j-1}
	\left\|
	\sum_{L=1}^{N_jQ}
	r_L b_i\lambda_{s_i}r_L^*
	\right\|_\varphi
	&\leq
	\sum_{i=1}^{j-1}\sum_{k=1}^{N_j}
	\left\|
	\sum_{l=1}^Q
	q_l b_i\lambda_{s_i}q_l^*p_k
	\right\|_\varphi\\
	&<\frac{\varepsilon}{6}.
\end{align*}
At the same time, we also observe that
\begin{align*}
\left\|\sum_{L=1}^{N_jQ}r_Lb_j\lambda_{s_j} {r_L}^*\right\|_\varphi&=
\left\|\sum_{k=1}^{N_j}\sum_{l=1}^Q p_k q_l b_j\lambda_{s_j} q_l^* 
p_k\right\|_\varphi\\& = \left\|\sum_{k=1}^{N_j}p_k\left(\sum_{l=1}^Q q_l 
b_j\alpha_{s_j}q_l^*\right)\lambda_{s_j} p_k\right\|_\varphi\\ &=
\left\|\sum_{k=1}^{N_j}p_k x_j\lambda_{s_j} 
p_k\right\|_\varphi\stackrel{(\ref{equ:secmainthmind})}{<}\frac{\varepsilon}{6} .
\end{align*}
Hence, using the triangle inequality, we obtain that
\begin{align*}
\left\|\sum_{L=1}^{N_jQ}r_L \left(\sum_{i=1}^j 
b_i\lambda_{s_i}\right){r_L}^*\right\|_\varphi=\left\|\sum_{k=1}^{N_j}\sum_{l=1}^Qp_k q_l \left(\sum_{i=1}^j 
b_i\lambda_{s_i}\right)q_l^* p_k\right\|_\varphi<
\frac{\varepsilon}{3}\numberthis \label{secondofthethreeind} .
\end{align*}
Furthermore,
\begin{align*}
\left\|\sum_{L=1}^{N_jQ}r_L\mathbb{E}(x){r_L}^*-\mathbb{E}(x)\right\|_\varphi &=
\left\|\sum_{k=1}^{N_j}\sum_{l=1}^Q p_k q_l\mathbb{E}(x)q_l^* p_k-\mathbb{E}(x)\right\|_\varphi \\& \leq 
\left\|\sum_{k=1}^{N_j}p_k\left(\sum_{l=1}^Q 
q_l\mathbb{E}(x)q_l^*\right)p_k-\sum_{l=1}^Q 
q_l\mathbb{E}(x)q_l^*\right\|_{\varphi} \\ & + 
\left\|\sum_{l=1}^{Q}q_l\mathbb{E}(x)q_l^*-\mathbb{E}(x)\right\|_\varphi\\
&\stackrel{(\ref{equ:secmainthm2ind}),(\ref{equ:firstmainthmparttwoind})}{<}
\frac{\varepsilon}{6}+\frac{\varepsilon}{6}=\frac{\varepsilon}{3} . \numberthis  
\label{thirdofthethreeind}
\end{align*}

This completes the induction. Taking $j=m$ and $\varepsilon=\delta/3$,
choose $q_1,\ldots,q_Q$ satisfying
\eqref{equ:inductionhypothesisone} and
\eqref{equ:inductionhypothesistwo}, and set
$x_{\delta,\varphi}:=\sum_{l=1}^Qq_lxq_l^*\in C$. The map
$T(z):=\sum_{l=1}^Qq_lzq_l^*$ is completely positive and contractive because
$\sum_lq_lq_l^*\leq1$. Then

\begin{align*}
&\left\|x_{\delta,\varphi}-\mathbb{E}(x)\right\|_\varphi\\ &=\left\|\sum_{l=1}^Q 
q_lxq_l^*-\mathbb{E}(x)\right\|_\varphi \\&=\left\|\sum_{l=1}^Q 
q_l\left(x-\sum_{i=1}^mb_i\lambda_{s_i}-\mathbb{E}(x)\right)q_l^*
+\sum_{l=1}^Qq_l\left(\sum_{i=1}^mb_i\lambda_{s_i}\right)q_l^*
+\sum_{l=1}^Qq_l\left(\mathbb{E}(x)\right)q_l^*-\mathbb{E}(x)\right\|_\varphi\\
&\leq \left\|x-\left(\sum_{i=1}^m 
b_i\lambda_{s_i}\right)-\mathbb{E}(x)\right\|+ 
\left\|\sum_{l=1}^Q q_l\left(\sum_{i=1}^m 
b_i\lambda_{s_i}\right)q_l^*\right\|_\varphi +
\left\|\sum_{l=1}^Q q_l\mathbb{E}(x)q_l^*-\mathbb{E}(x)\right\|_\varphi \\ &<\delta .
\end{align*}

We have shown that, for every state $\varphi$ on $D$ and every $\delta>0$,
there exists $x_{\delta,\varphi}\in C$ such that
\[
\|x_{\delta,\varphi}-\mathbb E(x)\|_\varphi<\delta.
\]
Suppose, toward a contradiction, that $\mathbb E(x)\notin C$. Since $C$ is a
norm-closed linear subspace of $D$, the Hahn--Banach theorem gives a functional
$f\in D^*$ such that $f|_C=0$ and $f(\mathbb E(x))=1$. Let
$f=v|f|$ be the polar decomposition of $f$ in $D^{**}$
\cite[Chapter~III, Section~4]{TakesakiI}, so that
$f(a)=|f|(v^*a)$ for every $a\in D$. Set $\psi:=|f|/\|f\|$. Then $\psi$ is a
state on $D$, and the Cauchy--Schwarz inequality gives
\[
|f(a)|^2\leq |f|(v^*v)\,|f|(a^*a)
\leq \|f\|^2\psi(a^*a),\qquad a\in D.
\]
Thus $|f(a)|\leq\|f\|\,\|a\|_\psi$ for every $a\in D$. Let
$\delta=(2\|f\|)^{-1}$, and choose $x_{\delta,\psi}\in C$ such that
\[
\|x_{\delta,\psi}-\mathbb E(x)\|_\psi<\delta.
\]
Then
\[
1=|f(\mathbb E(x)-x_{\delta,\psi})|
\leq\|f\|\,\|\mathbb E(x)-x_{\delta,\psi}\|_\psi<\frac12,
\]
a contradiction. Hence $\mathbb E(x)\in C$. Since $x\in C$ was arbitrary,
$\mathbb E(C)\subseteq C$. It now follows from  \cite[Proposition~3.4]{Suz17} that $C = \mathbb{E}(C)\rtimes_{\alpha,r} \G$.
\end{proof}

\section{Examples}\label{Examples}

This section presents several natural classes of examples satisfying the
hypotheses of Theorem \ref{mainintmresult}. {Reaching concrete
examples requires some preparation, and we now outline how the section is
organized. The hypotheses of \cref{mainintmresult} consist of two parts: the
central freeness of the induced action on the compressed algebras $PA^{**}P$,
and the central sequence condition $(PA P)'\cap (PA^{**}P)^\omega= (PB
P)'\cap (PA^{**}P)^\omega$. We first provide sufficient conditions for central
freeness: this is done in \cref{prop_topologically_free} for algebras with
Hausdorff spectrum on which the induced action is free, and in
\cref{cor_typeI} and \cref{rem_typeI} for type I algebras admitting a suitable
$\Gamma$-invariant composition series. We then exhibit, in
\cref{lemma:automatic-central-sequence-inclusion}, three classes of inclusions
$A\subseteq B$ for which the central sequence condition is automatic. The
reader interested in the resulting explicit applications of
\cref{mainintmresult} may wish to skip ahead to \cref{cor_examples}, in which
these ingredients are combined.}

It is well known (see \cite[Theorem 6.2.11]{gk}) that if $A$ is a type I 
$C^*$-algebra, then $A$ contains an essential ideal that has a continuous 
trace. Moreover, $A$ has a composition series 
\[
\{I_\theta\,|\, 0\leq \theta\leq \beta\}
\]
in which each quotient $I_{\theta+1}/I_\theta$ has Hausdorff spectrum (for 
$\theta <\beta$). Furthermore, if  
$\alpha\colon\Gamma\to\text{Aut}(A)$ is an action whose induced action on 
$\hat{A}$ is free, and if the above 
composition series is $\Gamma$-invariant, then 
the induced action of $\Gamma$ on $PA^{**}P$ is centrally free. Here, $P$ is a 
central $\alpha^{**}$-invariant projection in $B^{**}$ for which $PB^{**}P$ has 
a separable predual.

We now show that whenever $\hat{A}$ is
 Hausdorff and $\alpha$ induces a free action on $\hat{A}$, then for any 
 central 
$\alpha^{**}$-invariant projection $P$, the cut‑down action on $PA^{**}P$
is centrally free.

\begin{prop}\label[proposition]{prop_topologically_free}
	Let $\Gamma$ be a countable discrete group, and let $A$ be a separable 
	$C^*$-algebra with Hausdorff spectrum $\hat A$.  Suppose 
	\(\alpha\colon\Gamma\to \aut(A)\) induces a free action on $\hat A$.  
	Then for every central $\alpha^{**}$-invariant projection 
	$P\in A^{**}$ {such that $PA^{**}P$ has separable predual}, the action $\alpha^{**}$ on 
	$PA^{**}P$ is centrally free.  
\end{prop}

\begin{proof}
Fix $g\neq e$, and let $R$ be a nonzero central $\alpha_g^{**}$-invariant
projection in $PA^{**}P$. Suppose that the restriction of $\alpha_g^{**}$ to
$RA^{**}R$ is centrally trivial. For every $f\in C_0(\widehat A)$, the
constant sequence determined by $R(f1_{A^{**}})$ belongs to $(RA^{**}R)_\omega$.
It follows that
\[
R\bigl(f-\alpha_g(f)\bigr)1_{A^{**}}=0.
\]
Since the action on $\widehat A$ is free, the functions
$f-\alpha_g(f)$, for $f\in C_0(\widehat A)$, have no common zero. Hence the
closed ideal that they generate is all of $C_0(\widehat A)$, and therefore
$RC_0(\widehat A)1_{A^{**}}=0$. But $C_0(\widehat A)$ acts nondegenerately on
$A$, so an approximate identity of $C_0(\widehat A)$ converges strongly to
$1_{A^{**}}$. Thus $R=0$, a contradiction.
\end{proof}

\begin{cor}\label[corollary]{cor_typeI}
	Let \(\alpha\colon\Gamma\to\aut(B)\) be an action of a countable discrete 
	group on a separable \(C^*\)-algebra \(B\).  
	Let \(A\subseteq B\) be a $\Gamma$-invariant, separable, type~I 
	\(C^*\)-subalgebra which contains an approximate identity for \(B\).  
	Suppose further that \(A\) admits a $\Gamma$-invariant composition series  
	\[
	\{I_\theta\mid 0\le\theta\le\beta\},
	\]
	each quotient \(I_{\theta+1}/I_\theta\) (for \(\theta<\beta\)) having 
	Hausdorff spectrum, and that the induced action on \(\hat A\) is free.  
	Then for every central \(\alpha^{**}\)-invariant projection \(P\in B^{**}\) 
	{such that $PB^{**}P$ has separable predual, the action
$\alpha^{**}$ on $PA^{**}P$ is centrally free.}  
\end{cor}

\begin{proof}
Because $A$ contains an approximate identity for $B$, the map
\[
\pi_P\colon A^{**}\longrightarrow PA^{**}P,\qquad \pi_P(a)=PaP,
\]
is a normal, equivariant, surjective $*$-homomorphism. Its kernel is a
weak$^*$-closed ideal, so $PA^{**}P$ is equivariantly isomorphic to a central
summand of $A^{**}$.

For $\theta<\beta$, set $J_\theta:=I_{\theta+1}/I_\theta$. Repeatedly using
$A^{**}=J^{**}\oplus(A/J)^{**}$, together with continuity at limit ordinals,
gives the equivariant central decomposition
\[
A^{**}\cong\bigoplus_{\theta<\beta}J_\theta^{**}.
\]
Consequently, $PA^{**}P$ is a central direct sum of central corners of the
algebras $J_\theta^{**}$. Each nonzero such corner has separable predual, and
the induced action on the spectrum of $J_\theta$ is free. Hence
\cref{prop_topologically_free} shows that the action on each summand is
centrally free. A central direct sum of centrally free actions is centrally
free, which proves the result.
\end{proof}

Recall that a \emph{generalized continuous trace algebra} is a type I $C^*$-algebra which has a
canonical composition series (\cite[Section 10]{dixmier_1963}). Dixmier showed in
\cite{dixmier_1966} that the group $C^*$-algebra of any connected nilpotent Lie group belongs
to this class.

\begin{cor}\label[corollary]{rem_typeI}
{Let $\alpha\colon\Gamma\to\aut(B)$ be an action of a countable
discrete group on a separable $C^*$-algebra $B$. Let $A\subseteq B$ be a
$\Gamma$-invariant, separable, type~I $C^*$-subalgebra which contains an
approximate identity for $B$, and suppose that the induced action on $\hat A$
is free. Furthermore, suppose that one of the following two conditions holds:}
	\begin{enumerate}
		\item\label{rem-item-GTC} $\Gamma$ is any discrete group, and $A$ is a 
		generalized continuous trace algebra.
		\item\label{rem-item-finite-group}  $\Gamma$ is finite and $A$ is a 
		type I $C^*$-algebra.
		\end{enumerate}
Then, for every central $\alpha^{**}$-invariant projection $P$ in $B^{**}$ such that $PB^{**}P$ has separable predual, the action  $\alpha^{**}$ on $PA^{**}P$ is centrally free.
\end{cor}
\begin{proof}
	{By \cref{cor_typeI}, it suffices} to show that there exists a
		$\Gamma$-invariant composition series
		$\{I_\theta\,|\,0\leq\theta\leq\beta \}$, such that 
		$I_{\theta+1}/I_{\theta}$, 
		for $\theta<\beta$ has Hausdorff spectrum. For (\ref{rem-item-GTC}), 
		this is immediate, because there is a canonical composition series. For 
		(\ref{rem-item-finite-group}), this holds more 
		generally when $\Gamma$ is a compact group 
		(\cite[Corollary~8.1.2]{phillips1984k}). For the convenience of the 
		reader, we 
		provide a short proof when $\Gamma$ is finite.
		By  \cite[Theorem 6.2.11]{gk}, $A$ contains an essential ideal $I$ such 
		that 
		$I$ has continuous trace and hence has Hausdorff spectrum. Define 
		$J:=\cap_{g\in\Gamma} \alpha_g(I)$. Then $J$ is a nonzero closed 
		two-sided 
		ideal in $A$ having a Hausdorff spectrum. This gives the first step of 
		our 
		composition series. Repeat this construction in $A/J$ and continue by
		transfinite induction. We do not know whether one can always find 
		$\Gamma$-invariant composition series of this form if $\Gamma$ is
		infinite.
\end{proof}

{Next, we provide} examples for which the hypothesis
$(PA P)'\cap (PA^{**}P)^\omega=  (PB P)'\cap (PA^{**}P)^\omega$, for any
central projection $P$ in $B^{**}$ such that {$PB^{**}P$ has
separable predual}, is automatic.
\begin{lem}
	\label[lemma]{lemma:automatic-central-sequence-inclusion}
Let $A\subseteq B$ be an inclusion of separable $C^*$-algebras such that $A$ 
contains an approximate identity for $B$. Then the following conditions imply 
that for any central projection $P$ in $B^{**}$ such that
{$PB^{**}P$ has separable predual}, $(PA 
	P)'\cap (PA^{**}P)^\omega=  (PB P)'\cap (PA^{**}P)^\omega$.
\begin{enumerate}
	\item \label{C(X)algebra}$A$ is a subalgebra of the center of $B$
{(so $A\cong C_0(X)$ for a locally compact Hausdorff space $X$, and $B$ is
a $C_0(X)$-algebra)}.
	\item\label{tensor} $B$ is any $C^*$-algebra completion of  the algebraic tensor product of $C \otimes_{\text{alg}} D$, where $D$ is unital, and $A = C \otimes 1_{D}$.
	\item\label{new} Suppose $A$ is a separable unital $C^*$-algebra, let $H$ be a 
	countable discrete group, let $\beta \colon H \to \text{Aut}(A)$ be a pointwise 
	inner action, and assume $B = A \rtimes_{\beta,r} H$.
\end{enumerate}
\end{lem}
 (Notice that the third condition  
does not imply that the action is inner, so $B$ need not be isomorphic 
to $A \otimes_{\min} C^*_r(H)$; for instance, there is an example 
of an action of $\bZ_2 \times \bZ_2$ on $M_2$ such that the crossed
product is isomorphic to $M_4${; indeed, one can take the
action generated by $\Ad\left(\begin{smallmatrix} 0&1\\1&0
\end{smallmatrix}\right)$ and $\Ad\left(\begin{smallmatrix} 1&0\\0&-1
\end{smallmatrix}\right)$. As these two unitaries anticommute, they define a
projective representation of $\bZ_2\times\bZ_2$ on $\bC^2$ with a
nondegenerate $2$-cocycle $\sigma$, and a direct computation shows that
$M_2\rtimes(\bZ_2\times\bZ_2)\cong M_2\otimes
C^*_\sigma(\bZ_2\times\bZ_2)\cong M_2\otimes M_2$}.)

\textit{Proof of Part~(\ref{C(X)algebra})}:
This is immediate.

\textit{Proof of Part~(\ref{tensor})}:  Let $\varphi$ be a faithful normal 
state on $PB^{**}P$. By Proposition \ref{prop_ando_kirchberg}, in order to show 
$(PA 
	P)'\cap (PA^{**}P)^\omega\subseteq (P(A\otimes D) P)'\cap (PA^{**}P)^\omega$, it suffices to show 
$(PAP)'\cap (PAP,\varphi|_{PAP})_\omega\subseteq (P(A\otimes D)P)'\cap (PAP,\varphi|_{PAP})_\omega$. To this end, let $(c_n)_n\in\ell^{\infty}(\bN,PAP)$ such that 
\[
 \lim_{n \to \omega}\|c_na-ac_n\|^\#_{\varphi|_{PAP}} = 0,\,\,\,\forall a\in A .
\]
Pick $(z_1,z_2,\hdots) \in \ell^{\infty}(\bN,A)$ such that $c_n=Pz_nP$ for $z_n\in A$.
Let $y\in P(A\otimes D)P${. We show}
\[
{\lim_{n \to \omega} \|c_ny-yc_n\|^\#_{\varphi|_{P(A\otimes D)P}} = 0} .
\]
For this, set $M:=\sup\{\|z_n\| , n \in \bN\}$, which we may assume is nonzero; 
otherwise there is nothing to prove. Take $\varepsilon>0$ and let 
$\sum_{j=1}^k(a_j\otimes d_j)\in A\otimes  D$ such that 
\[
\left \|y-P\sum_{j=1}^k(a_j\otimes d_j)P \right \|<\frac{\varepsilon}{4M} .
\]
Now take
{$L:=\max\bigl(\{1\}\cup\{\|d_j\|:1\leq j\leq k\}\bigr)$ and
$S:=\max\bigl\{1,\sup_{n\in\bN,\,1\leq j\leq k}
\|c_na_j-a_jc_n\|^{1/2}\bigr\}$.} Now choose $W \in \omega$ such that for all $n\in W$ and for  $j=1,2,\hdots,k$, we have
\[
\|c_na_j-a_jc_n\|^{1/2}_{\varphi|_{PAP}}<\frac{\varepsilon}{2kSL} .
\]
Note that for $x\otimes y\in A\otimes D$ we have
\begin{equation*}
\begin{split}
 \|P(x\otimes y)P\|_{\varphi|_{P(A\otimes D)P}}^2 &=|\varphi(P(x^*x\otimes y^*y)P)|\\
 &=|\varphi(P(1\otimes y^*y)P P(x^*x\otimes 1)P)|\\
 &\leq \varphi(P(x^*xx^*x\otimes 1)P)^{1/2}\varphi(P(1\otimes y^*yy^*y)P)^{1/2}\\
 &\leq \|P(x\otimes 1)P\|\|y\|^2\|P(x\otimes 1)P\|_{\varphi|_{PAP}} .
 \end{split}
 \end{equation*}
 For any $n\in W$ we have 
\begin{equation*}
\begin{split}
\|[c_n,y]\|_{\varphi|_{P(A\otimes D)P}} &\leq \|[c_n,y-P\sum_{j=1}^k(a_j\otimes d_j)P]\|+\|[c_n,P\sum_{j=1}^k(a_j\otimes d_j)P]\|_{\varphi|_{P(A\otimes D)P}}\\
&\leq 2\|c_n\|\|y-P\sum_{j=1}^k(a_j\otimes d_j)P\|+\|[c_n,P\sum_{j=1}^k(a_j\otimes d_j)P]\|_{\varphi|_{P(A\otimes D)P}}\\
&\leq \frac{\varepsilon}{2}+\|[c_n,P\sum_{j=1}^k(a_j\otimes d_j)P]\|_{\varphi|_{P(A\otimes D)P}} ,
\end{split}
\end{equation*}
and 
\begin{equation*}
\begin{split}
\|[c_n,P\sum_{j=1}^k(a_j\otimes d_j)P]\|_{\varphi|_{P(A\otimes D)P}} &\leq \|\sum_{j=1}^k P((z_na_j-a_jz_n)\otimes d_j)P\|_{\varphi|_{P(A\otimes D)P}}
\\
&\leq \sum_{j=1}^k\|P((z_na_j-a_jz_n)\otimes d_j)P\|_{\varphi|_{P(A\otimes D)P}}\\
&\leq \sum_{j=1}^k \|c_na_j-a_jc_n\|^{1/2}\|d_j\|\|c_na_j-a_jc_n\|_{\varphi|_{PAP}}^{1/2}\\
&\leq \frac{\varepsilon}{2} .
\end{split}
\end{equation*}
Thus, for all $n\in W$ we have $\|[c_n,y]\|_{\varphi|_{P(A\otimes D)P}}<\varepsilon$. Similarly, one gets $\|[c_n,y]^*\|_{\varphi|_{P(A\otimes D)P}} \underset{n \to \omega}{\longrightarrow} 0$, hence $\|[c_n,y]\|_{\varphi|_{P(A\otimes D)P}}^\# \underset{n \to \omega}{\longrightarrow}  0$.

\textit{Proof of Part~(\ref{new})}: Since $\beta$ acts by unitary conjugation, for each $s\in H$, there 
exists a unitary $u_s\in  A $ such that 
$\beta_s(a)=u_sau_s^*$ for $a\in A $. By \cref{prop_ando_kirchberg}, as in the proof of Part~(\ref{tensor}), it suffices to show that for any faithful normal state $\varphi$ on $P(A\rtimes_{\beta,r} H)^{**}P$ with $P$ a central projection, we have 
\[
(PAP)'\cap (PAP,\varphi|_{PAP})_\omega\subset (P(A\rtimes_{\beta,r} H)P)'\cap(PAP,\varphi|_{PAP})_\omega .
\]
Let $(c_n)_n$ be a sequence in $\ell^{\infty}(\bN,PAP)$ such that $\|c_na-ac_n\|_{\varphi|_{PAP}}^\#  \underset{n \to \omega}{\longrightarrow} 0$ for all $a\in A$. Pick $(z_1,z_2,\hdots) \in \ell^{\infty}(\bN,A)$ such that $c_n=Pz_nP$ for $z_n\in A$.
Let $y\in  P(A \rtimes_{\beta,r} H)P$. We show that 
$\|c_ny-yc_n\|_{\varphi|_{P(A\rtimes_{\beta,r} H)P}}^\# \underset{n \to 
\omega}{\longrightarrow} 0$. Because $(c_n)$ is bounded, it suffices to prove this convergence when $y$ ranges over a dense subset of $P(A\rtimes_{\beta,r} H)P$. To this end, fix nonzero elements
$e_1,e_2,\ldots,e_k\in  A $, $s_1,s_2,\ldots,s_k\in H$ and $\varepsilon>0$.

For $l=1,2,\hdots,k$, define states $\varphi_l$ on $P(A\rtimes_{\beta,r} H)^{**}P$ by $\varphi_l(z)=\varphi(P\lambda
(s_l)^*z\lambda_{s_l}P)$. Because $\|\cdot\|_\varphi^\#$ generates the *-strong 
topology on the unit ball of $P(A\rtimes_{\beta,r} H)^{**}P$, there exists $W 
\in \omega$, such that for all $n\in W$ and for all $l=1,2,\hdots k$, we have
\[
\|c_ne_l-e_lc_n\|_{\varphi_l|_{PAP}}<\frac{\varepsilon}{2k} 
\,\,\text{and}\,\,\|u_{s_l}^*c_n-c_nu_{s_l}^*\|_{\varphi_l|_{PAP}}<\frac{\varepsilon}{2\sup\{\|e_l\|\}k}\,.
\]Then for all {$n\in W$}, we have
\begin{equation*}
\begin{split}
\|[c_n,P\sum_{l=1}^k e_l\lambda_{s_l}P]\|_{\varphi|_{P(A\rtimes_{\beta,r} H)P}} 
&=\|\sum_{l=1}^k 
P(z_ne_l\lambda_{s_l}-e_l\lambda_{s_l}z_n)P\|_{\varphi|_{P(A\rtimes_{\beta,r} 
H)P}}\\
&=\|\sum_{l=1}^k P( 
z_ne_l\lambda_{s_l}-e_l\beta_{s_l}(z_n)\lambda_{s_l})P\|_{\varphi|_{P(A\rtimes_{\beta,r}
 H)P}}\\
&=\|\sum_{l=1}^k 
P(z_ne_l\lambda_{s_l}-e_lu_{s_l}z_nu_{s_l}^*\lambda_{s_l})P\|_{\varphi|_{P(A\rtimes_{\beta,r}
 H)P}}\\
&\leq\sum_{l=1}^k\|P(z_ne_l-e_lz_n)P 
P\lambda_{s_l}P\|_{\varphi|_{P(A\rtimes_{\beta,r} H)P}}\\ 
&+\sum_{l=1}^k\|e_l\|\|P(u_{s_l}^*z_n-z_nu_{s_l}^*)P 
P\lambda_{s_l}P\|_{\varphi|_{P(A\rtimes_{\beta,r} H)P}}\\ 
&=\sum_{l=1}^k\|c_ne_l-e_lc_n\|_{\varphi_l|_{PAP}}+\sum_{l=1}^k\|e_l\|\|u_{s_l}^*c_n-c_nu_{s_l}^*\|_{\varphi_l|_{PAP}}\\
&<\varepsilon .
\end{split}
\end{equation*}
A similar calculation shows that $\|[c_n,P\sum_{l=1}^k 
e_l\lambda_{s_l}P]^*\|_{\varphi|_{P(A\rtimes_{\beta,r} H)P}} \underset{n \to 
\omega}{\longrightarrow}  0$.
\qed

{We end this section by combining the above results. The
following corollary provides a concrete class of examples to which
\cref{mainintmresult} applies.}

\begin{cor}\label[corollary]{cor_examples}
Let $\Gamma$ be a countable discrete group with property (AP). Let $D$ be a
separable unital $C^*$-algebra, 
and let $A$ be a separable type I $C^*$-algebra. Let $B:=A\otimes D$, and let
$\alpha:\Gamma\to\text{Aut}(B)$ be an action. {Assume that $\alpha$
leaves the canonical copy $A=A\otimes1_D$ invariant.} Suppose that either 
\begin{itemize}
	\item $\Gamma$ is 
	finite, or
	\item  $A$ is a generalized continuous trace algebra (see  
	\cite[Section 10]{dixmier_1963}).
\end{itemize}
 Suppose that the induced
action of $\Gamma$ on $\hat{A}$ is free. Then {the conclusion
of \cref{mainintmresult} holds:} any
intermediate subalgebra
\[
A \rtimes_{\alpha,r} \Gamma \subseteq C \subseteq B \rtimes_{\alpha,r} \Gamma
\]
 is of the form $C = B_1  \rtimes_{\alpha,r} \Gamma$, where $A \subseteq B_1 
 \subseteq B$ is an intermediate $\G$-invariant subalgebra.
\end{cor}

\begin{remark}
One can push the previous corollary a bit further. For concreteness, assume
$\Gamma$ is a finite group. Let $D$ be a separable unital $C^*$-algebra,
and let $A$ be a separable $C^*$-algebra. Let $B:=A\otimes D$ {and let
$\alpha:\Gamma\to\text{Aut}(B)$ be an action. Suppose that
$A_1\subseteq A_2\subseteq\cdots\subseteq A$ are type~I $C^*$-subalgebras
with $A=\overline{\bigcup_{n\in\mathbb N}A_n}$, and that both the canonical
copy of $A_n$ and $A_n\otimes D$ are $\alpha$-invariant for every $n$.}
Furthermore, assume the induced action of $\Gamma$ on $\hat{A}_n$ is free for
each $n\in\mathbb{N}$. Then any intermediate subalgebra
$A \rtimes_{\alpha,r} \Gamma \subseteq C \subseteq
B \rtimes_{\alpha,r} \Gamma$ is a crossed product obtained from
$\Gamma$ acting on an intermediate subalgebra. Indeed, fix $x\in C$, a state
$\varphi$ on $B\rtimes_{\alpha,r}\Gamma$, and $\delta>0$. Choose a finite
Fourier approximation of $x$ and then choose $n$ so that its coefficients and
$\mathbb E(x)$ are approximated sufficiently closely by elements of
$A_n\otimes D$. Apply the construction in the proof of
\cref{mainintmresult} to the inclusion $A_n\subseteq A_n\otimes D$. The
elements $q_l$ obtained there lie in $A_n\subseteq A$, so
$\sum_lq_lxq_l^*\in C$, and the estimates in that proof give an element
$z\in C$ with $\|z-\mathbb E(x)\|_\varphi<\delta$. The Hahn--Banach
argument at the end of the proof of \cref{mainintmresult} then gives
$\mathbb E(x)\in C$.
 \end{remark}

\section{Relative Rokhlin Dimension}\label{sec_rokhlin}

We introduce here the notion of \emph{relative Rokhlin dimension} for an 
inclusion $A\subseteq B$ of $\Gamma$-$C^*$-algebras, and show that such 
inclusions are centrally $\Gamma$-free in the sense of \cite{Suz}. This allows 
one to use {known results} about Rokhlin dimension to verify central 
$\Gamma$-freeness. We also provide a norm version of
Theorem~\ref{mainintmresult}, namely \cref{thm_normversion}:
by this we mean that the hypotheses of \cref{mainintmresult}, which concern
the induced action on the double dual and involve approximations in weak
topologies, are replaced by the requirement that the action on the inclusion
$A\subseteq B$ have pointwise finite relative Rokhlin dimension, a condition
in which the analogous approximations hold in norm.

We begin by recalling the definition of Rokhlin dimension for actions of finite groups and 
of $\bZ$ from \cite{hirshberg2015rokhlin,hirshberg_phillips}. For actions of $\bZ$, we use the single tower 
version.
\begin{definition}
	Let $A$ be a separable $C^*$-algebra. 
	\begin{enumerate}
		\item Suppose $\alpha \in \aut(A)$ (thought of as an action of $\bZ$). 
		We say that $\alpha$ has \emph{Rokhlin dimension} $d$ if $d$ is the 
		least 
		non-negative integer such that for any $\varepsilon>0$, for any finite 
		subset $F \subseteq A$ and for any $N\in \bN$ there exist positive 
		contractions $\{f_k^{(l)} \mid k=1,2,\ldots,N \, ; \, l=0,1,2,\ldots 
		d\}$ 
		in $A$ such that for all $x\in F$ we have
			\begin{enumerate}
				\item $\|(\sum_{l=0}^d \sum_{k=1}^N f_k^{(l)})x-x\|<\varepsilon$.
				\item For all $j \neq k$ and for all $l$ we have 
				$\|(f_k^{(l)}f_j^{(l)})x\| < \varepsilon$.
				\item For all $l$ and for all $k$ we have $\| (\alpha(f_k^{(l)}) 
				- f_{k+1}^{(l)})x\| < \varepsilon$ with the convention that 
				$f_{N+1}^{(l)} = 
				f_1^{(l)}$. 
				\item For any indices $k,l$ we have 
				$\| [ x,f_k^{(l)} ] \| < \varepsilon$. 
			\end{enumerate} 
		\item Suppose $G$ is a finite group and suppose $\alpha \colon G \to 
		\aut(A)$. 
		We say that $\alpha$ has \emph{Rokhlin dimension} $d$ if $d$ is the 
		least 
		non-negative integer such that for any $\varepsilon>0$ and for any finite 
		subset $F \subseteq A$ there exist positive 
		contractions $\{f_g^{(l)} \mid g \in G \, ; \, l=0,1,2,\ldots 
		d\}$ 
		in $A$ such that for all $x\in F$ we have
		\begin{enumerate}
			\item $\|(\sum_{l=0}^d \sum_{g \in G} f_g^{(l)})x-x\| <\varepsilon$.
			\item For all $g \neq h$ in $G$ and for all $l$ we have 
			$\| (f_g^{(l)}f_h^{(l)})x \| < \varepsilon$.
			\item For all $l$ and for all $g,h \in G$ we have 
			$\| (\alpha_g(f_h^{(l)}) -	f_{gh}^{(l)})x \| < \varepsilon$.
			 \item For any index $l$ and for any $g 
			\in G$ we have 
			$\| [ x,f_g^{(l)} ] \| < \varepsilon$. 
		\end{enumerate} 
	\end{enumerate}
\end{definition}
For this paper, it {suffices} to consider the restriction to any singly generated subgroup, so we make the following definition.
\begin{definition}
	Let $A$ be a separable $C^*$-algebra. 
	 Let $\Gamma$ be a countable discrete group, and suppose $\alpha \colon \Gamma \to 
	\aut(A)$ is an action. We say that $\alpha$ has 
	\emph{pointwise finite Rokhlin dimension} if for any $s \in \Gamma 
	\smallsetminus \{e\}$ the restriction of $\alpha$ to the subgroup generated 
	by $s$ has a finite Rokhlin dimension. 
\end{definition}
We now introduce a notion of Rokhlin dimension for an action on an inclusion $A \subseteq B$. The key point here is that we want the Rokhlin towers to be in $A$ but approximately central with respect to elements of $B$. 
\begin{definition}
	Let $A \subseteq B$ be an inclusion of separable $C^*$-algebras. 
	\begin{enumerate}
		\item Suppose $\alpha \in \aut(B)$ is an automorphism which leaves $A$ invariant (thought of as an action of $\bZ$). 
		We say that $\alpha$ has \emph{relative Rokhlin dimension} $d$ if $d$ 
		is the least 
		non-negative integer such that for any $\varepsilon>0$, for any finite 
		subset $F \subseteq B$ and for any $N\in \bN$ there exist positive 
		contractions $\{f_k^{(l)} \mid k=1,2,\ldots,N \, ; \, l=0,1,2,\ldots 
		d\}$ 
		in $A$ such that for all $x\in F$ we have
		\begin{enumerate}
			\item $\|(\sum_{l=0}^d \sum_{k=1}^N f_k^{(l)})x-x\| <\varepsilon$.
			\item For all $j \neq k$ and for all $l$ we have 
			$\|(f_k^{(l)}f_j^{(l)})x\| < \varepsilon$.
			\item For all $l$ and for all $k$ we have $\| (\alpha(f_k^{(l)}) 
			- f_{k+1}^{(l)})x\| < \varepsilon$ with the convention that 
			$f_{N+1}^{(l)} = 
			f_1^{(l)}$. 
			\item For any indices $k,l$ we have 
			$\| [ x,f_k^{(l)} ] \| < \varepsilon$. 
		\end{enumerate} 
		\item Suppose $G$ is a finite group and suppose $\alpha \colon G \to 
		\aut(B)$ is an action which leaves $A$ invariant. 
		We say that $\alpha$ has \emph{relative Rokhlin dimension} $d$ if $d$ 
		is the least 
		non-negative integer such that for any $\varepsilon>0$ and for any finite 
		subset $F \subseteq B$ there exist positive 
		contractions $\{f_g^{(l)} \mid g \in G \, ; \, l=0,1,2,\ldots 
		d\}$ in $A$ such that for all $x\in F$ we have
		\begin{enumerate}
			\item $\|(\sum_{l=0}^d \sum_{g \in G} f_g^{(l)})x-x\| <\varepsilon$.
			\item For all $g \neq h$ in $G$ and for all $l$ we have 
			$\| (f_g^{(l)}f_h^{(l)})x \| < \varepsilon$.
			\item For all $l$ and for all $g,h \in G$ we have 
			$\| (\alpha_g(f_h^{(l)}) -	f_{gh}^{(l)})x \| < \varepsilon$.
			 \item For any index $l$ and for any $g 
			\in G$ we have 
			$\| [ x,f_g^{(l)} ] \| < \varepsilon$. 
		\end{enumerate} 
	\item Let $\Gamma$ be a countable discrete group, and suppose $\alpha \colon \Gamma \to 
	\aut(B)$ is an action which leaves $A$ invariant. We say that $\alpha$ has 
	\emph{pointwise finite relative Rokhlin dimension} if for any $s \in \Gamma 
	\smallsetminus \{e\}$ the restriction of $\alpha$ to the subgroup generated 
	by $s$ has a finite relative Rokhlin dimension. 
	\end{enumerate}
\end{definition}

\begin{remark}\label[remark]{rem_fainfb}
	Suppose $A \subseteq B$ is an inclusion of separable $C^*$-algebras such that the natural inclusion from $A_{\infty}$ into $B_{\infty}$ descends to a unital homomorphism $F(A)\subseteq F(B)$ {(see Section~\ref{Sec: Prelim} for the definition of $F(A)$)}. It is immediate that an action on $B$, which leaves $A$ invariant, has pointwise finite relative Rokhlin dimension provided the restriction to $A$ has a pointwise finite Rokhlin dimension.  In addition, it is easy to verify that the conditions mentioned in Lemma~\ref{lemma:automatic-central-sequence-inclusion} guarantee that we have such a unital homomorphism $F(A) \subseteq F(B)$. 
\end{remark}

The following is a sufficient condition for pointwise finite Rokhlin dimension.
\begin{lem}\label[lemma]{lem_gnsouter}
	Let $A$ be a separable unital nuclear simple stably finite $\cZ$-stable $C^*$-algebra such that the trace space $T(A)$ is a Bauer simplex with finite-dimensional extreme boundary. Let $\Gamma$ be a discrete group and let $\alpha \colon \Gamma \to \aut(A)$ be an action such that for any $s \in \Gamma \smallsetminus \{e\}$ and for any $\alpha_s$-invariant trace $\tau$, the extension of $\alpha_s$ to the GNS closure $\overline{A}^{\|\cdot\|_{\tau}}$ is outer. Then, the action $\alpha$ has pointwise finite Rokhlin dimension.
\end{lem}
\begin{proof}
	This follows immediately from the combination of \cite[Theorems A, B]{GHV} and \cite[Theorem C]{wouters}.
\end{proof}

Next, we {recall} the notion of centrally $\Gamma$-free inclusion for the inclusion $A\subseteq B$ of $\Gamma$-$C^*$-algebras{, which was introduced in \cite[Definition~4.1]{Suz},} and show that if the inclusion of $\Gamma$-$C^*$-algebras, $A\subseteq B$, has  pointwise finite relative Rokhlin dimension, then the inclusion is centrally $\Gamma$-free.

For a $C^*$-algebra $A$, define
\[
\ell^{2}(A):=\Big\{{\xi}=({\xi}_i)_{i=1}^\infty\in\ell^\infty(\mathbb{N},A)\colon
 \text{the series}\, \sum_{i=1}^\infty {\xi}_i^*{\xi}_i\,\text{converges in  
norm}\Big\}.
\]
We equip $\ell^{2}(A)$ with the norm $\lVert {\xi}\rVert_2:=\lVert\sum_{i=1}^\infty {\xi}_i^*{\xi}_i\rVert^{1/2}$. When $A$ is a $\Gamma $-$C^*$-algebra, we equip $\ell^{2}(A)$ with the pointwise $\Gamma$-action. For ${\xi}=({\xi}_i)_{i=1}^\infty, {\eta}=({\eta}_i)_{i=1}^\infty\in\ell^{2}(A)$, $c\in A$, we define
\[
{\xi}^*{\eta}:=\sum_{i=1}^\infty {\xi}_i^*{\eta}_i\in A,
\]
\[
{\xi}c:=({\xi}_ic)_{i=1}^\infty,\,\, c{\xi}:=(c{\xi}_i)_{i=1}^\infty\in\ell^{2}(A).
\]
\begin{definition}[{{\cite[Definition~4.1]{Suz}}}]\label[definition]{def_centfree}
Let $\Gamma$ be a countable discrete group and  $A\subseteq B$ be an inclusion of $\Gamma$-$C^*$-algebras. We say that the inclusion is \emph{centrally $\Gamma$-free} if for any $b,c\in B$, $\varepsilon>0$, and $s\in\Gamma\setminus\{e\}$, there is an element ${\xi}\in\ell^{2}(A)_1$ satisfying $\lVert {\xi}^*b{\xi}-b\rVert<\varepsilon$, and $\lVert {\xi}^*cs({\xi})\rVert<\varepsilon$.
\end{definition}

\begin{lemma}\label[lemma]{lem_suz}
Let $\Gamma$ be a countable discrete group. Let $\alpha:\Gamma\to\aut(B)$ be an 
action which leaves $A$ invariant. Furthermore, assume that the inclusion 
$A\subseteq B$ of $\Gamma$-$C^*$-algebras has  pointwise finite relative 
Rokhlin dimension. Then, the inclusion is centrally $\Gamma$-free.
\end{lemma}

\begin{proof}
{
Assume that the inclusion $A\subseteq B$ has pointwise finite relative
Rokhlin dimension, and take $b,c\in B$, $\varepsilon>0$, and
$s\in\Gamma\setminus\{e\}$. Condition~(a) in the definition implies
$\overline{AB}=B$; taking adjoints also gives $\overline{BA}=B$. Hence every
contractive approximate identity of $A$ is an approximate identity for $B$.
Choose a positive contraction $e\in A$ such that
\[
\|e^{1/2}be^{1/2}-b\|<\varepsilon/4.
\]
Let $N$ be the order of $s$ if it is finite, or any integer greater than $2$ if
$s$ has infinite order, and let $d$ be the relative Rokhlin dimension of the
restriction to the subgroup generated by $s$. Set $m:=(d+1)N$ and choose
\[
0<\delta<\min\left\{1,\frac{\varepsilon}{8m},
\frac{\varepsilon}{4(1+\|b\|)}\right\}.
\]
Set $F:=\{b,c,b^*,c^*,e^{1/2}\}$. Choose positive contractions
$\{f_k^{(l)}:1\leq k\leq N,\ 0\leq l\leq d\}$ satisfying the relative
Rokhlin-dimension conditions with tolerance $\delta$ for $F$. By decreasing
the initial tolerance and using polynomial approximation to the square-root
function, we may also assume, for every $x\in F$, that
\begin{enumerate}
\item $\|(\sqrt{f_k^{(l)}}\sqrt{f_j^{(l)}})x\|<\delta$ whenever $j\neq k$;
\item $\|(\alpha_s(\sqrt{f_k^{(l)}})-\sqrt{f_{k+1}^{(l)}})x\|<\delta$;
\item $\|[x,\sqrt{f_k^{(l)}}]\|<\delta$,
\end{enumerate}
with the usual cyclic convention. Let $h:=\sum_{l,k}f_k^{(l)}$. Since
$e^{1/2}\in F$,
\[
\|e^{1/2}he^{1/2}-e\|<\delta,
\qquad \|e^{1/2}he^{1/2}\|\leq1+\delta.
\]
Define
\[
\zeta:=\bigl(\sqrt{f_k^{(l)}}e^{1/2}\bigr)_{l,k}\in\ell^2(A),
\qquad r:=(1+\delta)^{-1/2},
\qquad {\xi}:=r\zeta.
\]
Then
\[
\|\xi\|_2^2
=r^2\|e^{1/2}he^{1/2}\|\leq1,
\]
so ${\xi}\in\ell^2(A)_1$. As in the usual Rokhlin-tower
estimate,
\begin{align*}
\left\|\sum_{l,k}\sqrt{f_k^{(l)}}c\,
\alpha_s(\sqrt{f_k^{(l)}})\right\|<3m\delta.
\end{align*}
Therefore
\begin{align*}
\|{\xi}^*c\,s({\xi})\|
&=r^2\left\|e^{1/2}\left(\sum_{l,k}\sqrt{f_k^{(l)}}c\,
\alpha_s(\sqrt{f_k^{(l)}})\right)\alpha_s(e^{1/2})\right\|\\
&<3m\delta<\varepsilon.
\end{align*}
Similarly,
\[
\left\|\sum_{l,k}\sqrt{f_k^{(l)}}b\sqrt{f_k^{(l)}}-b\right\|
<2m\delta,
\]
and hence
\[
\|\zeta^*b\zeta-b\|<2m\delta+\varepsilon/4<\varepsilon/2.
\]
Finally,
\begin{align*}
\|{\xi}^*b{\xi}-b\|
&\leq r^2\|\zeta^*b\zeta-b\|+(1-r^2)\|b\|\\
&<\varepsilon/2+\delta\|b\|<\varepsilon.
\end{align*}
Thus the inclusion is centrally $\Gamma$-free.
}
\end{proof}
The following result is the norm version of \cref{mainintmresult}.

\begin{theorem}\label[theorem]{thm_normversion}
Let $\Gamma$ be a countable discrete group with property AP. Let $B$ be a
separable $C^*$-algebra, and let $\alpha \colon \Gamma \to \aut(B)$ be an 
action. Let $A$ be a $\Gamma$-invariant subalgebra of $B$. 
Assume that the action of $\Gamma$ on the inclusion $A \subseteq B$ has pointwise finite relative Rokhlin dimension.
Then, any intermediate $C^*$-algebra $C$ with $A\rtimes_{\alpha,r}  \G\subset 
C\subset B\rtimes_{\alpha,r}  \G$ is of the form 
$B_1\rtimes_{\alpha,r}  \Gamma$, where $B_1$ is a $\Gamma$-invariant 
$C^*$-subalgebra of $B$.
\end{theorem}
\begin{proof}
The conclusion follows directly from \cref{lem_suz} and \cite[Theorem 
4.5]{Suz}. 
\end{proof}

We end this section with an example which is a consequence of the above result, 
\cref{lem_gnsouter} and \cref{rem_fainfb}.

\begin{example}
	Let $D$ be a separable unital $C^*$-algebra. Let $A$ be a separable nuclear $C^*$-algebra, and let $B = A \otimes D$. Let $\Gamma$ be a countable discrete group {with property (AP)}, and let $\alpha \colon \Gamma \to \aut(B)$ be an action that leaves $A$ invariant. Suppose $A$ is a separable unital nuclear simple stably finite $\cZ$-stable $C^*$-algebra such that the trace space $T(A)$ is a Bauer simplex with finite-dimensional extreme boundary and for any $s \in \Gamma \smallsetminus \{e\}$ and for any $\alpha_s$-invariant trace $\tau$, the extension of $\alpha_s$ to the GNS closure $\overline{A}^{\|\cdot\|_{\tau}}$ is outer. Then, any intermediate subalgebra $A \rtimes_{\alpha,r} \Gamma \subseteq C \subseteq B \rtimes_{\alpha,r} \Gamma$ is a crossed product. {For an explicit example, let $\Gamma$ be any countably infinite amenable group, let $A=\bigotimes_{g\in\Gamma}M_2\cong M_{2^\infty}$, and let $\alpha=\sigma\otimes\beta$, where $\sigma$ is the Bernoulli shift action of $\Gamma$ on $A$ and $\beta$ is an arbitrary action of $\Gamma$ on $D$. The algebra $A$ is separable, unital, nuclear, simple, stably finite and $\cZ$-stable, and $T(A)$ consists of a unique trace $\tau$; moreover, the extension of $\sigma_s$, for $s\in\Gamma\smallsetminus\{e\}$, to the GNS closure $\overline{A}^{\|\cdot\|_{\tau}}$, which is the hyperfinite II$_1$ factor, is well known to be outer (that is, the Bernoulli shift action is strongly outer; see, for instance, \cite{GHV}).}
\end{example}

\bibliographystyle{plain}
\bibliography{Cd}

@article{AK,
	author = {Amrutam, Tattwamasi and Kalantar, Mehrdad},
	doi = {10.1017/etds.2019.34},
	journal = {Ergodic Theory and Dynamical Systems},
	number = {12},
	pages = {3181--3187},
	publisher = {Cambridge University Press},
	title = {On simplicity of intermediate {$C^*$}-algebras},
	volume = {40},
	year = {2020}}

@book{BroOza08,
	author = {Brown, Nathanial P. and Ozawa, Narutaka},
	pages = {xvi+509},
	publisher = {American Mathematical Society, Providence, RI},
	series = {Graduate Studies in Mathematics},
	title = {{{$C^*$}}-algebras and finite-dimensional approximations},
	volume = {88},
	year = {2008}}

@article{Suz,
	author = {Suzuki, Yuhei},
	journal = {Communications in Mathematical Physics},
	number = {2},
	pages = {1273--1297},
	title = {Complete descriptions of intermediate operator algebras by intermediate extensions of dynamical systems},
	volume = {375},
	year = {2020}}

@article{Suz17,
	author = {Suzuki, Yuhei},
	journal = {American Journal of Mathematics},
	number = {3},
	pages = {681-705},
	title = {Group {$C^*$}-algebras as decreasing intersection of nuclear 
	{$C^*$}-algebras},
	vol = {139},
	year = {2017}}

@article{weiss2001monotileable,
	author = {Weiss, Benjamin},
	journal = {Translations of the American Mathematical Society-Series 2},
	pages = {257--262},
	publisher = {Providence [etc.] American Mathematical Society, 1949-},
	title = {Monotileable amenable groups},
	volume = {202},
	year = {2001}}

@article {connes1975outer,
	AUTHOR = {Connes, Alain},
	TITLE = {Outer conjugacy classes of automorphisms of factors},
	JOURNAL = {Annales Scientifiques de l'{\'E}cole Normale Sup{\'e}rieure},
	FJOURNAL = {Annales Scientifiques de l'\'{E}cole Normale Sup\'{e}rieure. 
	Quatri\`eme
	S\'{e}rie},
	VOLUME = {8},
	YEAR = {1975},
	NUMBER = {3},
	PAGES = {383--419},
	ISSN = {0012-9593},
	MRCLASS = {46L10},
	MRREVIEWER = {Hisashi Choda},
	URL = {http://www.numdam.org/item?id=ASENS_1975_4_8_3_383_0},
}

@article {GHV,
	AUTHOR = {Gardella, Eusebio and Hirshberg, Ilan and Vaccaro, Andrea},
	TITLE = {Strongly outer actions of amenable groups on
	{$\mathcal{Z}$}-stable nuclear {$C^*$}-algebras},
	FJOURNAL = {J. Math. Pures Appl. (9)},
	JOURNAL = {Journal de Math\'{e}matiques Pures et Appliqu\'{e}es.
	Neuvi\`eme S\'{e}rie},
	VOLUME = {162},
	YEAR = {2022},
	PAGES = {76--123},
	ISSN = {0021-7824,1776-3371},
	MRCLASS = {46L55 (37A55 46L35 46L40)},
	MRREVIEWER = {Daniel\ Gon\c{c}alves},
	DOI = {10.1016/j.matpur.2022.04.003},
	URL = {https://doi.org/10.1016/j.matpur.2022.04.003},
}

@article{hirshberg2012decomposable,
	author = {Hirshberg, Ilan and Kirchberg, Eberhard and White, Stuart},
	journal = {Advances in Mathematics},
	number = {3},
	pages = {1029--1039},
	publisher = {Elsevier},
	title = {Decomposable approximations of nuclear {$C^*$}-algebras},
	volume = {230},
	year = {2012}}

@article {hirshberg_phillips,
		AUTHOR = {Hirshberg, Ilan and Phillips, N. Christopher},
		TITLE = {Rokhlin dimension: obstructions and permanence properties},
		FJOURNAL = {Doc. Math.},
		JOURNAL = {Documenta Mathematica},
		VOLUME = {20},
		YEAR = {2015},
		PAGES = {199--236},
		ISSN = {1431-0635,1431-0643},
		MRCLASS = {46L55},
		MRREVIEWER = {G\'{a}bor\ Szab\'{o}},
	}

@book{ocneanu2006actions,
	author = {Ocneanu, Adrian},
	publisher = {Springer},
	title = {Actions of discrete amenable groups on von {N}eumann algebras},
	volume = {1138},
	year = {1985}}

@article{cameron2016intermediate,
	author = {Cameron, Jan Michael and Smith, Roger R},
	journal = {International Journal of Mathematics},
	number = {11},
	pages = {1650091},
	publisher = {World Scientific},
	title = {Intermediate subalgebras and bimodules for general crossed 
	products of von {N}eumann algebras},
	volume = {27},
	year = {2016}}

@article{Cameron2015BimodulesIC,
	author = {Jan Michael Cameron and Roger R. Smith},
	journal = {Advances in Mathematics},
	pages = {539-561},
	title = {Bimodules in crossed products of von {N}eumann algebras},
	volume = {274},
	year = {2015}}

@article{A19,
	author = {Amrutam, Tattwamasi},
	journal = {International Mathematics Research Notices},
	pages = {16193--16204},
	title = {{On intermediate subalgebras of {$C^*$}-simple group actions}},
	year = {2021}}

@article{izumi1998galois,
	author = {Izumi, Masaki and Longo, Roberto and Popa, Sorin},
	journal = {Journal of Functional Analysis},
	number = {1},
	pages = {25--63},
	publisher = {Elsevier},
	title = {A {G}alois correspondence for compact groups of automorphisms of 
	von 
	{N}eumann algebras with a generalization to {K}ac algebras},
	volume = {155},
	year = {1998}}

@article{choda1978galois,
	author = {Choda, Hisashi},
	journal = {Tohoku Mathematical Journal, Second Series},
	number = {4},
	pages = {491--504},
	publisher = {Mathematical Institute, Tohoku University},
	title = {A {G}alois correspondence in a von {N}eumann algebra},
	volume = {30},
	year = {1978}}

@article {rordam2021irreducible,
		AUTHOR = {R{\o}rdam, Mikael},
		TITLE = {Irreducible inclusions of simple {$C^*$}-algebras},
		FJOURNAL = {Enseign. Math.},
		JOURNAL = {L'Enseignement Math\'{e}matique},
		VOLUME = {69},
		YEAR = {2023},
		NUMBER = {3-4},
		PAGES = {275--314},
		ISSN = {0013-8584,2309-4672},
		MRCLASS = {46L05 (46L35 46L55)},
		DOI = {10.4171/lem/1051},
		URL = {https://doi.org/10.4171/lem/1051},
	}

@article{haagerup1994approximation,
	author = {Haagerup, Uffe and Kraus, Jon},
	journal = {Transactions of the American Mathematical Society},
	number = {2},
	pages = {667--699},
	title = {Approximation properties for group {$C^*$}-algebras and group von 
	{N}eumann algebras},
	volume = {344},
	year = {1994}}

@article{ryo2021remark,
	author = {Ochi, Ryo},
	journal = {Hokkaido Mathematical Journal},
	number = {2},
	pages = {247--262},
	publisher = {Hokkaido University, Department of Mathematics},
	title = {A remark on the freeness condition of {S}uzuki's correspondence 
	theorem for intermediate {$C^*$}-algebras},
	volume = {50},
	year = {2021}}

@article{ozawa2008weak,
	author = {Ozawa, Narutaka},
	journal = {Groups, Geometry, and Dynamics},
	number = {2},
	pages = {271--280},
	title = {Weak amenability of hyperbolic groups},
	volume = {2},
	year = {2008}}

@book{TakesakiI,
	author = {Takesaki, M.},
	pages = {xix+415},
	publisher = {Springer-Verlag, Berlin},
	series = {Encyclopaedia of Mathematical Sciences},
	title = {Theory of operator algebras. {I}},
	volume = {124},
	year = {2002}}

@book{TakesakiIII,
	author = {Takesaki, M.},
	pages = {xxii+548},
	publisher = {Springer-Verlag, Berlin},
	series = {Encyclopaedia of Mathematical Sciences},
	title = {Theory of operator algebras. {III}},
	volume = {127},
	year = {2003}}

@article {ornstein1980ergodic,
	AUTHOR = {Ornstein, Donald S. and Weiss, Benjamin},
	TITLE = {Ergodic theory of amenable group actions. {I}. {T}he {R}ohlin
	lemma},
	FJOURNAL = {Bull. Amer. Math. Soc. (N.S.)},
	JOURNAL = {Bulletin of the American Mathematical Society. New Series},
	VOLUME = {2},
	YEAR = {1980},
	NUMBER = {1},
	PAGES = {161--164},
	ISSN = {0273-0979,1088-9485},
	MRCLASS = {28D20 (49A50)},
	MRREVIEWER = {Douglas\ Lind},
	DOI = {10.1090/S0273-0979-1980-14702-3},
	URL = {https://doi.org/10.1090/S0273-0979-1980-14702-3},
}

@article{ornstein1987entropy,
	author = {Ornstein, Donald S and Weiss, Benjamin},
	journal = {Journal d'Analyse Math{\'e}matique},
	number = {1},
	pages = {1--141},
	publisher = {Springer},
	title = {Entropy and isomorphism theorems for actions of amenable groups},
	volume = {48},
	year = {1987}}

@phdthesis{jarrett2018non,
	author = {Jarrett, Kieran},
	school = {University of Bath},
	title = {Non-singular actions of countable groups},
	year = {2018}}

@article{hirshberg2015rokhlin,
	author = {Hirshberg, Ilan and Winter, Wilhelm and Zacharias, Joachim},
	journal = {Communications in Mathematical Physics},
	number = {2},
	pages = {637--670},
	publisher = {Springer},
	title = {{R}okhlin dimension and {$C^*$}-dynamics},
	volume = {335},
	year = {2015}}

@article{ando2016non,
  title={Non-commutativity of the central sequence algebra for separable non-type {I} {$C^*$}-algebras},
  author={Ando, Hiroshi and Kirchberg, Eberhard},
  journal={Journal of the London Mathematical Society},
  volume={94},
  number={1},
  pages={280--294},
  year={2016},
  publisher={Oxford University Press}
}

@book {gk,
	AUTHOR = {Pedersen, Gert K.},
	TITLE = {{$C^* $}-algebras and their automorphism groups},
	SERIES = {London Mathematical Society Monographs},
	VOLUME = {14},
	PUBLISHER = {Academic Press, Inc. [Harcourt Brace Jovanovich, Publishers],
	London-New York},
	YEAR = {1979},
	PAGES = {ix+416},
	ISBN = {0-12-549450-5},
	MRCLASS = {46Lxx},
	MRREVIEWER = {J.\ W.\ Bunce},
}

@article {wouters,
	AUTHOR = {Wouters, Lise},
	TITLE = {Equivariant {$\mathcal{Z}$}-stability for single automorphisms on
	simple {$C^*$}-algebras with tractable trace simplices},
	FJOURNAL = {Math. Z.},
	JOURNAL = {Mathematische Zeitschrift},
	VOLUME = {304},
	YEAR = {2023},
	NUMBER = {1},
	PAGES = {Paper No. 22, 36},
	ISSN = {0025-5874,1432-1823},
	MRCLASS = {46L55 (46L40)},
	DOI = {10.1007/s00209-023-03278-7},
	URL = {https://doi.org/10.1007/s00209-023-03278-7},
}

@book {phillips1984k,
	AUTHOR = {Phillips, N. Christopher},
	TITLE = {Equivariant {$K$}-theory and freeness of group actions on
	{$C^*$}-algebras},
	SERIES = {Lecture Notes in Mathematics},
	VOLUME = {1274},
	PUBLISHER = {Springer-Verlag, Berlin},
	YEAR = {1987},
	PAGES = {viii+371},
	ISBN = {3-540-18277-2},
	MRCLASS = {46L80 (19L47 22D25 46L55 46M20 58G12)},
	MRREVIEWER = {Judith\ A.\ Packer},
	DOI = {10.1007/BFb0078657},
	URL = {https://doi.org/10.1007/BFb0078657},
}

@article {dixmier_1963,
	AUTHOR = {Dixmier, Jacques},
	TITLE = {Traces sur les {$C\sp*$}-alg\`ebres},
	FJOURNAL = {Ann. Inst. Fourier (Grenoble)},
	JOURNAL = {Universit\'e{} de Grenoble. Annales de l'Institut Fourier},
	VOLUME = {13},
	YEAR = {1963},
	PAGES = {219--262},
	ISSN = {0373-0956,1777-5310},
	MRCLASS = {46.65},
	MRREVIEWER = {J.\ M. G. Fell},
	URL = {http://www.numdam.org/item?id=AIF_1963__13_1_219_0},
}

@article {dixmier_1966,
	AUTHOR = {Dixmier, Jacques},
	TITLE = {Sur le dual d'un groupe de {L}ie nilpotent},
	FJOURNAL = {Bull. Sci. Math. (2)},
	JOURNAL = {Bulletin des Sciences Math\'ematiques. 2e S\'erie},
	VOLUME = {90},
	YEAR = {1966},
	PAGES = {113--118},
	ISSN = {0007-4497},
	MRCLASS = {22.60},
	MRREVIEWER = {A.\ Kleppner},
}

\end{document}